\documentclass[11pt,
]{amsart}

\usepackage{
amssymb, graphicx
}

\usepackage{xcolor,cite}
 
\date{\today}

\numberwithin{equation}{section}

\newtheorem{theorem}{Theorem}[section]
\newtheorem{proposition}{Proposition}[section]
\newtheorem{lemma}{Lemma}[section]

\theoremstyle{definition}
\newtheorem{remark}{Remark}[section]

\DeclareMathOperator{\Vol}{Vol}

\DeclareMathOperator{\WF}{WF}

\newcommand{\eps}{\varepsilon}

\newcommand{\PP}{P}
\newcommand{\SH}{S \hspace{-.1em} H}
\newcommand{\SV}{SV}

\newcommand{\R}{\mathbf{R}}
\newcommand{\C}{\mathbf{C}}
\newcommand{\Id}{\textrm{\rm Id}}
\renewcommand{\r}[1]{(\ref{#1})}
\newcommand{\PDO}{{\rm $\Psi$DO}}
\newcommand{\be}[1]{\begin{equation}\label{#1}}
\newcommand{\ee}{\end{equation}}

\newcommand{\p}{{\partial}}
\newcommand{\foliation}{\mathsf{x}}

\renewcommand{\d}{\mathrm{d}}

\renewcommand{\i}{\mathrm{i}}

\newcommand{\bo}{\partial \Omega}

\newcommand{\bl}{\mathrm{b}}
\newcommand{\loc}{\mathrm{loc}}

\newcommand{\Hbloc}{H_{\bl,\loc}}

\newcommand{\gzero}{g}

\title[The  transmission problem in linear isotropic elasticity]{The  transmission problem in linear isotropic elasticity}

\author[P. Stefanov]{Plamen Stefanov}
\address{Department of Mathematics, Purdue University, West Lafayette, IN 47907}
\thanks{First author partly supported by  NSF  Grant DMS-1600327}

\author[G. Uhlmann]{Gunther Uhlmann}
\address{Department of Mathematics, University of Washington, Seattle, WA 98195, Department of Mathematics University of Helsinki, Finland, IAS, HKUST, Clear Water Bay, Hong Kong, China}
\thanks{Second author partly supported by NSF Grant CMG-1025259 and DMS-1265958,  and a Simons fellowship}

\author[A. Vasy]{Andras Vasy}
\address{Department of Mathematics, Stanford University, Stanford  CA 94305}
\thanks{Third author partly supported by  DMS-1664683.}

\begin{document}
\begin{abstract} 
We study the isotropic elastic wave equation in a bounded domain with boundary with coefficients having jumps at a nested set of interfaces satisfying the natural transmission conditions there. We analyze in detail the microlocal behavior of such solution like reflection, transmission and mode conversion of S and P waves, evanescent modes, Rayleigh and Stoneley waves. In particular, we recover Knott's equations in this setting. 
We show that   knowledge of the Dirichlet-to-Neumann map determines uniquely the speed of the P and the S waves  if there is a strictly convex foliation with respect to them,   under an additional condition of lack of full internal reflection of some of the waves. 
\end{abstract} 
\maketitle

\section{Introduction}  
The main goal of this work is to study 
the transmission problem in  isotropic linear elasticity. Let $\Omega\subset\R^3$ be a smooth  bounded domain. 
Let $\Gamma_1,\dots, \Gamma_k$ be closed disjoint smooth surfaces (interfaces) splitting $\Omega$ into subdomains $\Omega_k$ with exterior boundary $\Gamma_{k-1}$ (with $\Gamma_0:=\bo$) and interior one $\Gamma_k$, see Figure~\ref{pic_el_waves}, left.    Assume that the density $\rho$ and the Lam\'e  parameters $\mu$, $\nu$ are smooth up to those surfaces with possible jumps there. We also assume that at at every point, at least one coefficient has a non-zero jump. We impose the following transmission conditions
\be{tr}
[u]= 0, \quad [Nu]=0 \quad\text{on $\Gamma_j$, $j=1,\dots,k$},
\ee
where $[v]$ stands for the jump of $v$ from the exterior to the interior  across any of those surfaces, and $Nf$ are the normal components of the stress tensor, see \r{2a1}. 
We are motivated by the isotropic elastic model of the Earth where the density and the Lam\'e parameters jump across the boundary between the crust and the mantle, etc. 
We study the time-dependent elastic system, see \r{el_eq}.

The first goal of this paper is to describe qualitatively  the microlocal behavior of solutions of this problem. 
At any interface $\Gamma_i$, an incoming S or P wave can generate two reflected waves, one S wave and one P wave through mode conversion and two transmitted ones. Then each branch can generate four more, etc., see Figure~\ref{pic_el_waves}. In some cases, there might be a full internal reflection for one or both of the waves, and there could be no transmitted or reflected waves of a certain kind. In fact, the missing waves would be evanescent modes. 

\begin{figure}[!ht]
\includegraphics[page=8,scale=1]{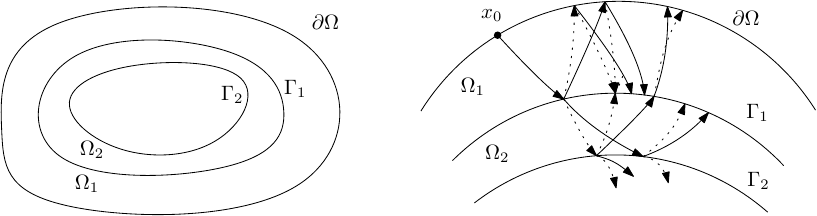}
\caption{Left: the domain $\Omega$ and the layers. Right: Propagation of rays from a single source and direction. P waves are denoted with a solid line; S waves are dotted.
}\label{pic_el_waves}
\end{figure}

While works   on geometric optics for the elasticity system  exist (no transmission) \cite{Brytik_11, HansenUhlmann03, Rachele_2000, Rachele00, Rachele03,SUV_elastic,  Bhattacharyya_18}, a comprehensive analysis of the transmission problem in linear elasticity has not been done to authors' knowledge. In case of a flat surface and constant coefficients, some cases have been analyzed in the geophysics literature, see, e.g.,  \cite{Petrashen_seismic_I, Petrashen_seismic_II, aki2002quantitative, Sheriff_Seismology, slawinski2003seismic}. In that case, if there is no full internal reflection, one looks for solutions in terms of potentials to reduce the number of variables; and the potentials of the four waves corresponding to an incoming one solve  a system which decouples into a $4\times4$ and a $2\times2$ one, see also \r{u_C4} and \r{u_C4a}. Those equations were derived by Knott \cite{Knott1899} and Zoeppritz \cite{Zoeppritz19} more than a century ago, see also \cite{aki2002quantitative}.  In a recent paper \cite{caday2019recovery}, the hyperbolic-hyperbolic (HH) case is analyzed for variable $\lambda(x)$, $\mu(x)$ and $\rho=1$ but the construction for a curved boundary is  partial only.  The (HH) case is characterized by the wavefront of the Cauchy data on $\Gamma$: it could belong to projected S and P waves on either side of it, and in particular, there are no evanescent modes, see Section~\ref{sec_el_bvp}. This is just one of the many cases since we may have full internal reflection of some or both waves on one or both sides of $\Gamma$; and mode conversion to evanescent modes, see Section~\ref{sec_Tr_summary} for a summary. The most general study we are aware of is \cite{Yamamoto_elastic_89} where the coefficients are constant but cases other than the (HH) one are considered, even though not as extensively as we do it in this paper.

We analyze the general case of variable coefficients and a curved interface in all cases, away from glancing rays. We are interested in two main questions: is the problem well posed microlocally; and  (control) can we create every configuration on one side with suitably  chosen waves on the other. By doing that, we also compute the principal parts of the reflected and the transmitted waves. The microlocal well posedness reduces to showing the ellipticity of some \PDO\ system on $\Gamma$ with not particularly simple looking entries. Its solution  serves as initial conditions for the corresponding transport equations for the hyperbolic of for the evanescent modes. In the flat, constant coefficients case, this system is actually the computation giving us the whole solution.  Going back to the general case, in the (HH) microlocal region, we have four outgoing waves, each one being  3D vector-valued. This gives as a $12\times12$ \PDO\ system for showing-well posedness. If we allow both S and P waves coming from both sides, we would have a $12\times 24$ system which we want to solve for some group of variables. The control question is reduced to solving the same system with a rearrangement of the unknowns: we are given the waves on one side and want to solve for the waves on the other. 

Doing this analysis with brute force does not seem to be a promising approach. Instead, we look for inspiration in the geophysics (and the existing math) literature using the flat constant coefficient case as a starting point. We express the P and the S waves in terms of potentials, as the divergence and as the curl of such potentials on a principal symbol level first; and we extend this to an arbitrary order.  We adapt this to the boundary value problem. Having such microlocal mode separation, we also split the S waves in the SV (shear-vertical) and SH (shear-horizontal) waves. This decomposition is valid on $\Gamma$ only, and depends on the point (and the codirection). Then we reduce those systems to more manageable decoupled $4\times4$ plus $2\times2$ ones for the outgoing solutions given the incoming ones;  their extended versions are $4\times 8$ plus $2\times 4$ ones, see \r{main_system} and \r{2x2}. If the boundary is flat and the coefficients are constant, those are exactly Knott's equations \cite{Knott1899}. Their ellipticity, needed to show well posedness, turns out to be a consequence of energy preservation (even though the determinant can be computed and analyzed \cite{aki2002quantitative}), another observation due to Knott. Ellipticity needed to show control can be verified easily and follows from the microlocal well posedness of the boundary value Cauchy problem. 

We do this analysis in all microlocal cases with some or even all waves being evanescent; in that case we call them modes.  The corresponding matrix symbols do not need to be recomputed; we just need to be careful which imaginary square roots to chose. Ellipticity based on energy preservation needs modifications though. Evanescent waves do not carry (high frequency) energy on the principal symbol level, at least. 

We do such analysis for the boundary value problem for the outgoing solutions as well with Dirichlet or Neumann, homogeneous or not, boundary conditions. We also analyze the microlocal  boundary value Cauchy problem. We start with the (principally scalar) acoustic equation first for two reasons: it is a needed ingredient in the analysis of the elastic system and SH waves behave as acoustic ones (no mode conversion). 

We also study the surface waves propagating along the boundary (Rayleigh waves) or along an internal interface $\Gamma$ (Stoneley waves).  Taylor \cite{Taylor-Rayleigh} characterized Rayleigh waves as a propagation of singularities phenomenon when $n=2$ and $\bo$ is flat, and he also mentions that the analysis applies to the general case as well. The existence of such waves is due to lack of ellipticity of the  Dirichlet-to-Neumann (DN) operator in the elliptic region on $\bo$ and in the elliptic-elliptic one on an internal interface. Restricted to the surface $\bo$ or $\Gamma$, they solve a real principal type of system; and the solution extends as an evanescent one in $\bar \Omega$.  Yamamoto \cite{Yamamoto_elastic_89} viewed Stoneley waves in a similar fashion. A more detailed analysis of the Rayleigh and the Stoneley waves will appear in a work of Y.~Zhang. 

We also present an application of this analysis to the inverse problem of recovering the coefficients form the outgoing DN  map. We recover first the lens relation associated with incoming S and P waves in the first layer $\Omega_1$; then we use the recent results by the authors  \cite{SUV_localrigidity} about local recovery of a sound speed (or a conformal factor) from localized travel times. By \cite{Bhattacharyya_18}, we can recover $\rho$ in $\Omega_1$ as well, therefore we can recover all three coefficients $\mu$, $\lambda$ and $\rho$ there. 
In  \cite{SUV_localrigidity} we prove conditional H\"older stability as well which makes this approach for the inverse problem in this paper potentially stable as well; when it can be applied. 
 In the case of no internal interfaces, this was done in \cite{SUV_elastic}. 
The inverse problem for  transversely anisotropic media  is studied in \cite{Vasy2019recovery}. The presence of  interfaces however complicates the geometry considerably, see  Figure~\ref{pic_el_waves} for the recovery of the coefficients in the deeper layers. The lens relation corresponding to a single S or P wave (ray) is multi-valued in general and there is no direct way to tell which branch is coming from which layer, roughly speaking. This makes the inverse problem much different. An essential difficulty following this approach is that there could be totally  internally reflected rays in the interior side of one interface which never get out, not even through mode conversion. Then they cannot be generated by rays from the exterior (by ``earthquakes''). We show that if there is no  total internal reflection of S waves on the interface $\Gamma_1$ (from the interior),  we can recover $c_s$ below it. This is more general than the result in \cite{caday2019recovery} where $\rho=1$, and there is the implicit assumption that there is no full reflection of S \textit{and} P waves. Since we do not recover all three coefficients below the first interface, we use arguments based on the geometry and the directions of the polarization only, which depend on the speeds only. Next, we also show that if there is no  total internal reflection of $P$ waves as well, one can recover $c_p$ in $\Omega_2$. Those arguments can be used to get even deeper into $\Omega$ with the appropriate assumptions on the speeds.

\section{Preliminaries} 

\subsection{The elastic system} 
The isotropic elastic system in a smooth bounded domain $\Omega\subset \R^3$ is described as follows. The elasticity  tensor is defined by
\[
c_{ijkl} = \lambda \delta_{ij}  \delta_{kl} +\mu(\delta_{ik}\delta_{jl} + \delta_{il}\delta_{jk}),
\]
where $\lambda$, $\mu>0$ are the Lam\'e parameters. 
Assume for now that the coefficients $\lambda$, $\mu$ and $\rho$ are smooth in $\bar \Omega$.  The elastic wave operator is given by
\[
(Eu)_i = \rho^{-1}\sum_{jkl} \partial_j c_{ijkl} \partial_l u_k,
\]
where $\rho>0$ is the density and the vector function $u$ is the displacement.  
The corresponding elastic wave equation is given by 
\be{el_eq}
u_{tt}-Eu=0,
\ee
see, e.g., \cite{slawinski2003seismic}. 
The stress tensor $\sigma_{ij}(u)$ is defined by
\be{1s0}
\sigma_{ij}(u) = \lambda \nabla\cdot u\delta_{ij} + \mu(\partial_j u_i + \partial_i u_j). 
\ee
Note that $Eu = \rho^{-1}\delta\sigma(u)$, where $\delta$ is the divergence of the 2-tensor $\sigma(u)$. 

The Dirichlet boundary condition for $E$ is prescribing $u$ on the boundary; while the natural Neumann boundary condition is to prescribe the normal components of the stress tensor
\be{2a1}
 Nu:=  \sum_j \sigma_{ij}(u)\nu^j\big|_{\bo},
\ee
where $\nu$ is the outer unit normal on $\bo$. This is the operator appearing in the Green's formula \r{9G} for $E$ but also has the physical meaning as the infinitesimal deformation of the material in normal direction. 

Let $\Gamma$ be a smooth surface where the coefficients $\rho$, $\lambda$, $\mu$ may jump. The physical transmission conditions across $\Gamma$ are the following. First,  kinematic ones:   the displacements $u$ on both sides of $\Gamma$ should  match (no slipping of the material w.r.t.\ each other); and second, dynamical ones:  the normal components $Nu$ on both sides should match (same traction). Therefore, if we declare one side of $\Gamma$ external and the other one internal, and denote by $[u]_\Gamma$ the jump of $u$ across $\Gamma$ from the exterior  to the interior, we obtain the transmission conditions \r{tr} on $\Gamma$. Note that in $[Nu]$, the operator $N$ depends on $\rho$, $\mu$ and $\lambda$ and has different coefficients on each side of $\Gamma_j$.

The operator $E$ is symmetric on $L^2(\Omega;\mathbf{C}^3,\rho\,\d x)$. It has a principal symbol
\be{s0}
\sigma_p(-E)v  = \frac{\lambda+\mu}{\rho} \xi (\xi\cdot v) + \frac{\mu}{\rho} |\xi|^2 v,\quad v\in\mathbf{C}^n,
\ee
which can be also written as 
\be{s0'}
\sigma_p(-E)v  = \frac{\lambda+2\mu}{\rho} \xi (\xi\cdot v) + \frac{\mu}{\rho} \left(|\xi|^2 -\xi\xi\cdot\right)v. 
\ee

Taking $v=\xi$ and $v\perp\xi$, we recover the well known fact that $\sigma_p(-E)$ has eigenvalues $c_p^2$ and $c_s^2$ with 
\be{speeds} 
c_p= \sqrt{(\lambda+2\mu)/\rho}, \quad c_s = \sqrt{\mu/\rho}
\ee
of multiplicities $1$ and $2$ and eigenspaces $\R\xi$, and $\xi^\perp$, respectively. 
We have $c_s<c_p$. Those are known as the speeds of the P waves and the S waves, respectively.  
 The eigenspaces correspond to  the polarization of those waves. The characteristic variety $\det \sigma_p((\partial_t^2-E)) =0$ is the union of $\Sigma_p := \{\tau^2=c_p^2|\xi|^2\}$ and $\Sigma_s := \{\tau^2=c_s^2|\xi|^2\}$, each one  having two connected components (away from the zero section), determined by the sign of $\tau$.

Let $u$ solve the elastic wave equation 
\be{1}
\begin{cases}
u_{tt} -Eu &=0\quad \text{in $\R\times\Omega$},\\ 
 \ u|_{\R\times\bo}   &=f,\\
\ \ \ \  u|_{t<0}&=0,
\end{cases}
\ee
with $f$ given so that $f=0$ for $t<0$ and all coefficients smooth in $\Omega$ (no transmission interfaces). The (outgoing) Dirichlet-to-Neumann $\Lambda$ map is defined by
\be{2a}
(\Lambda f)_i = (Nu)_i =  \sum_j \sigma_{ij}(u)\nu^j\big|_{\bo},
\ee
see \r{2a1}, where $\nu$ is the outer unit normal on $\bo$, and $\sigma_{ij}(u)$ is the stress tensor \r{1s0}. 

%
%

\subsection{An invariant metric based formulation} \label{sec_m}
We have 
\be{E}
(Eu)_i = \rho^{-1} (  \partial_i\lambda\partial_j u_j + \partial_j\mu \partial_j u_i + \partial_j\mu \partial_i u_j     ),
\ee 
where we sum over repeating indices even if they are both lower or upper. 
This can also be written in the following divergence form 
\be{L0}
Eu = \rho^{-1}(  \d \lambda \delta u+ 2 \delta \mu \d^s u   ),
\ee
where $\d^su=(\partial_ju_i+ \partial_iu_j)/2$ is the symmetric differential, and $\delta= -(\d^s)^*$ is the divergence of symmetric fields with the  adjoint in $L^2$ sense.

To prepare ourselves for changes of variables needed in the analysis near surfaces that we will flatten out, we will write $E$ in an invariant way in the presence of a Riemannian metric $g$. We view $u$ as an one form (a covector field) and we define the symmetric differential $\d^s$ and the divergence $\delta$ by
\[
(\d^s u)_{ij}= \frac12\left(\nabla_i u_j+\nabla_j u_i\right), \quad (\delta v)_i = \nabla^j v_{ij},\quad \delta u = \nabla^iu_i,
\]
where $\nabla$ is the covariant differential, $\nabla^j = g^{ij}\nabla_i$, $u$ is a covector field, and  $v$ is a symmetric covariant tensor field of order two.   Note that $\d^s$ increases the order of the tensor by one while $\delta$ decreases it by one. Then we define $E$ by \r{L0}. We still have $\delta= -(\d^s)^*$, where the adjoint is in the $L^2(\Omega,\d\Vol)$ space of contravariant tensor fields, see, e.g., \cite{Sh-book}. 

The stress tensor \r{1s0} is given by
\be{1s2}
\sigma(u) = \lambda (\delta u)g + 2\mu \d^s u,
\ee
and then $Eu=\rho^{-1}\delta\sigma(u)$. 
The Neumann boundary condition $Nu$ at $\bo$ is still given by prescribing the values of $\sigma_{ij}(u)\nu^j$ on it as in \r{2a}.   
The operator $E$, defined originally on 
$C_0^\infty(\Omega)$  extends to a self-adjoint operator in $L^2(\Omega, \rho\, \d\!\Vol)$. This extension is the one satisfying the zero Dirichlet boundary condition on $\R\times\bo$. In particular, this shows that the mixed problem \r{1} is solvable with regular enough  data $f$ at least since one can always extend $f$ inside and reduce the problem to solving one with a zero boundary condition and a non-zero source term; and then use the Duhamel's principle for the latter.

 The principal symbol of $E$ in the metric setting is still given by \r{s0} with the proper interpretation of the dot product there:
\be{s0g}
(\sigma_p(-E)v)_i  = \frac{\lambda+\mu}{\rho} \xi_i \xi^j v_j + \frac{\mu}{\rho} |\xi|_g^2 v,\quad v\in\C^n,
\ee
where $\xi^j=g^{jk}\xi_k$ as usual. 
In particular, the speeds $c_p$ and $c_s$ remain as in \r{speeds}. The eigenspaces of the symbol are still $\R\xi$ and $\xi^\perp$, the latter being the covectors normal to $\xi$. 
 Notice that under coordinate changes, the coordinate expression for $u$ changes as well, as a covector. 

We recall that the cross product on an oriented three dimensional Riemannian manifold is defined in the following way. If $\xi$ and $\eta$ are covectors at some fixed point $x$, then $\xi\times \eta$ is defined as the unique covector satisfying
\[
\langle \xi\times \eta,\zeta\rangle = \omega(g^{-1} \xi,g^{-1}\eta,g^{-1}\zeta),
\]
where $\langle\cdot, \cdot\rangle$ is the metric inner product of covectors, and $\omega$ is the volume form on the  tangent bundle. To compute it in local coordinates, let $\alpha=\xi\times\eta$. Then we get
\[
g^{ij}\alpha_i\zeta_j = (\det g)^{-\frac12}\det(\xi,\eta,\zeta),
\]
where the latter is the determinant of the matrix with the indicated columns (also, the Euclidean volume form of them). Therefore, $(\det g)^{1/2}g^{-1}\alpha$ equals the Euclidean cross product 
\[
(\det g)^{1/2}g^{-1}\alpha = (\xi_2\eta_3-\xi_3\eta_2, -\xi_1\eta_3+\xi_3\eta_1, \xi_1\eta_2-\xi_2\eta_1).
\]
This yields
\be{cross}
\xi\times\eta = (\det g)^{-\frac12}g (\xi_2\eta_3-\xi_3\eta_2, -\xi_1\eta_3+\xi_3\eta_1, \xi_1\eta_2-\xi_2\eta_1).
\ee
Similarly, the curl $\nabla\times u$ of a covector field $u$ is defined as  the Hodge star of the exterior derivative $\d u $, and we have
\be{curl}
\nabla\times u = (\det g)^{-\frac12}g (\partial_2u_3-\partial_3 u_2, -\partial_1 u_3+\partial_3u_1, \partial_1 u_2-\partial_2u_1). 
\ee
The divergence of $u$ is given by $\delta u = \nabla^i u_i$ and in particular, $\delta\nabla\times u=0$. We will use the notation $\nabla \cdot u$  for $\delta u$ as well. 

One can verify that the  double vector product of two covectors in the metric still satisfies   $\xi\times \eta\times \zeta = \langle \xi,\zeta\rangle \eta -\langle \xi,\eta\rangle \zeta$, as in the Euclidean case. 

\subsection{Existence of dynamics}
We assume now, as in the rest of the paper, that $\Omega$ can be expressed as a union  of layers as explained in the Introduction and  $\lambda$, $\mu$ and $\rho$ are smooth up to their boundaries with possible jumps at them. We also assume that $E$ is the metric based operator \r{L0}.

\begin{lemma}
Let $\lambda,\mu,\rho$ be as above. Then $E$, defined originally on functions smooth up to $\Gamma_1,\dots \Gamma_k$ and $\bo$, satisfying the transmission conditions \r{tr}, and zero boundary conditions on $\bo$, extends to a self-adjoint operator in $L^2(\Omega, \rho\,\d\!\Vol)$.
\end{lemma}

\begin{proof} We start with Green's formula. Let $D$ be a bounded domain with a smooth boundary so that $\lambda,\mu,\rho$  are smooth in $\bar D$. Then 
\be{9G}
\int_D \langle Eu,v \rangle\rho \,\d\!\Vol- \int_D \langle u,Ev \rangle\rho\,\d\!\Vol = \int_{\p D} \left(  \langle Nu , v\rangle- \langle u, Nv \rangle \right)\d A,
\ee
where $\d A$ is the area measure in $\p D$ induced by $g$. 
To prove it, write 
\[
\int_D \langle Eu,v \rangle\rho \,\d\!\Vol
 = -\int_D \big(  \lambda \langle \delta u,\delta v\rangle  + 2\mu \langle\d^s u, \d^s v\rangle \big)\d\!\Vol+   \int_{\p D}\sigma_{ij}(u)\nu^j v^i \, \d A,
\]
since   $Eu=\rho^{-1}\delta\sigma(u) $.  
The last integral equals
\[
   \int_{\p D} \langle Nu,  v\rangle \, \d A.
\]
Switch $u$ and $v$ and subtract the resulting formulas to prove \r{9G}. 

Assume now that $u$ and $v$ are smooth up to the interfaces, may jump there and satisfy the transmission conditions \r{tr}. 
We apply \r{9G} to $\Omega\setminus \Omega_1$, $\Omega_1\setminus \Omega_2$, \dots, $\Omega_k$ and sum up the results. Note that the outer normal to $\Omega\setminus \Omega_1$ at $\Gamma_1$ is the inner one at the same $\Gamma_1$ when viewed from  $\Omega_1\setminus \Omega_2$, etc. As a result, we get \r{9G} in $\Omega$ as well, despite the discontinuities because by the transmission conditions \r{tr}, all contributions from $\Gamma_1,\dots,\Gamma_k$ cancel. By the zero boundary condition on $\bo$, the r.h.s.\ of \r{9G} vanishes. Therefore, $E$ is symmetric. 

To show that there is a natural self-adjoint extension, it is enough to show that the quadratic form $(-Eu,u)$ is bounded from below. For every smooth $u$ satisfying the Dirichlet boundary condition, by \r{L0} we have
\[
(-Eu,u)= \int_\Omega \left(  \lambda |\delta u|^2 + 2\mu |\d^s u|^2 \right)\d\!\Vol,
\]
which is non-negative. 

We can write the Cauchy problem at $t=0$ for \r{el_eq} with Dirichlet boundary conditions now as 
\[
\partial_t (u_1,u_2) = \mathbf{E}(u_1,u_2) := (u_2,Eu_1), \quad (u_1,u_2)|_{t=0}=(f_1,f_2). 
\]
The operator $\mathbf{E}$ is self-adjoint on the energy $\mathcal{H}$ space with norm
\[
\|(f_1,f_2)\|^2_\mathcal{H} = \int_\Omega \left(  \lambda |\delta f_1|^2 + 2\mu |\d^s f_1|^2 +|f_2|^2\right)\d\!\Vol.
\]
Then by Stone's theorem, the Cauchy problem at $t=0$ for \r{el_eq} with Dirichlet boundary conditions is solved by a unitary group. Problem \r{1} can be solved for regular enough $f$ by extending $f$ inside $\Omega$ and reducing it to a problem with a source but with homogeneous Dirichlet boundary conditions; and solving it by Duhamel's formula. 
\end{proof}

\subsection{The Neumann boundary operator}
 Let $x=(x',x^3)$ be semigeodesic coordinates  to a given surface $\Gamma$, with $x^3>0$ on one side of it, defining the orientation in the metric setup. 
The  metric then takes the form $g$ in those coordinates with $g_{\alpha 3}=\delta_{\alpha 3}$ for $1\le\alpha\le 3$. Then, see also \cite{SUV_elastic},
\[
(Nu)_j = \lambda (\delta u) \delta_{j3} + \mu\left( \partial_3 u_j + \partial_j u_3- 2 \Gamma_{j3}^ku_k\right).
\]
Therefore,
\be{Nu}
\begin{split}
 (Nu)_j  &= 
  \mu  (\partial_3u_{j}+ \partial_j u_{3})-2\mu \Gamma_{j3}^\nu u_\nu \ ,\quad j=1,2,\\
 (Nu)_3  &=   \lambda (\partial_1 u_{1} +\partial_2  u_{2}) +   (\lambda +2\mu)\partial_3 u_{3} ,
\end{split}
\ee
where $\nu=1,2$ and we used the fact that $\Gamma_{33}^k= \Gamma_{3k}^3=0$.

\section{Geometric optics for the wave equation with \PDO\  lower order terms} \label{sec_GO} 
We recall the well known geometric optics construction for a hyperbolic pseudo-differential  equation generalizing  the acoustic  wave equation, see, e.g., \cite{Taylor-book0, Treves}. We allow the equation to be a system but we still assume that the principal part is scalar, see also \cite{Dencker_polar}.  In this generality, the construction is done in \cite[VIII.3]{Taylor-book0}. We are not going to formulate results about the propagation of the polarization set which can be derived from  \cite{Dencker_polar}. The reason to do study the acoustic equation in this generality is two-fold. First, the elastic system decomposes into such pseudo-differential equations; and second, SH waves propagate like acoustic ones as we show below. 

\subsection{The Cauchy Problem with data at $t=0$} \label{sec_Ac_Cauchy}

Our interest is in  the acoustic wave equation with lower order classical  pseudo-differential term $A\in\Psi^1$
\be{ac}
(\partial_t^2- c^2 \Delta_{{\gzero}}+A)u=0
\ee
 with Cauchy data $(u,\partial_t u)=(h_1,h_2)$ at $t=0$. Here, ${\gzero}$ is a Riemannian metric that we include in order to have the flexibility to change coordinates easily; and $\Delta_{{\gzero}}$ is the Laplace-Beltrami operator. The distribution $u$ is vector valued and $A$ is a matrix valued \PDO. 
Up to lower order terms, $c^2\Delta_g$ coincides with $\Delta_{c^{-2}g}$. 
The characteristic variety $\Sigma$ is given by $\tau^2=c^2|\xi|_{{\gzero}}^2$ and has two connected components $\Sigma_\pm$  corresponding to $\tau<0$ and $\tau>0$, away from the zero section (notice the convention that $\tau<0$ corresponds to $\Sigma_+$). 
We are looking for solutions of the form 
\be{o1}
\begin{split}
u(t,x) =  (2\pi)^{-n} \sum_{\sigma=\pm}\int e^{\i\phi_\sigma(t,x,\xi)} &\Big( a_{1,\sigma}(t,x,\xi) \hat h_1(\xi)\\
&+  a_{2,\sigma}(t,x,\xi) |\xi|_{{\gzero}}^{-1}\hat h_2(\xi)\Big) \d \xi,
\end{split}
\ee
modulo  terms involving smoothing operators of $h_1$ and $h_2$, defined in some neighborhood of $t=0$, $x=x_0$ with some $x_0$.  This parametrix differs from the actual solution by a smoothing operator applied to $\mathbf{h}=(h_1,h_2)$, as it follows from standard hyperbolic estimates. The signs $\sigma=\pm$ correspond to solutions with wave front sets in $\Sigma_\mp$, respectively as it can be seen by applying the stationary phase lemma. 

Here, $a_{j,\sigma}$ are classical amplitudes of order zero depending smoothly on $t$ of the form
\be{a}
a_{j,\sigma} \sim \sum_{k=0}^\infty a_{j,\sigma}^{(k)},\quad \sigma=\pm, \; j=1,2,
\ee
where $a_{j,\sigma}^{(k)}$ is homogeneous in $\xi$ of degree $-k$ for large $|\xi|$. 
The phase functions $\phi_\pm$ are positively homogeneous of order $1$ in $\xi$ solving the eikonal equations
\be{o2}
\partial_t\phi\pm c(x)|\nabla_x\phi|_{{\gzero}}=0, \quad 
\phi_\pm|_{t=0}=x\cdot\xi.
\ee
Such solutions exist locally only, in general. While the principal symbol is the only one determining the eikonal equations and therefore the geometry, the subprincipal symbol in \r{ac} depending on the principal one of $A$, affects the leading amplitude below. 

Since the principal symbol of the hyperbolic operator in \r{ac} allows the decomposition  $-\tau^2+c^2|\xi|_{{\gzero}} = (-\tau+c|\xi|_{{\gzero}})(\tau+c|\xi|_{{\gzero}}) $, in a conic neighborhood of $\Sigma_+$, one can apply a parametrix of $D_t-c|D|_{{\gzero}}$ to write \r{ac} there as
\be{4.5}
(\partial_t+\i c|D|_{{\gzero}}+A_+)u_+=0\quad \text{mod $C^\infty$}
\ee
with $A_+$ of order zero and $u_+$ being the sum of the $\sigma=+$ terms in \r{o1}. This is the case studied in \cite[VIII.3]{Taylor-book0} with a more general elliptic $-\lambda(t,x,D)$ replacing $\i c|D|_{{\gzero}}+A_+$, allowing $u_+$ to be a vector function, and $A_+$ to be matrix valued.

The main tool is the ``fundamental lemma'' allowing us to understand the action of a \PDO\ $P$ on $e^{\i\phi}a$ in terms of a homogeneous expansion in $\xi$, see \cite[VIII.7]{Taylor-book0} and \cite{Treves2}. The lemma remains true for principally scalar systems and it is used for such in \cite{Taylor-book0}.

We recall the  construction of the amplitude. Let $u_+$ be as the first term in \r{o1} with the indices there dropped, corresponding to $\sigma=+$. We seek the amplitude of the form $a=a_0+a_1+\dots$ as in \r{a} but the upper index $(k)$ is a lower one now. 

The order two terms in   the expansion of $(\partial_t -\i \lambda(t,x,D))u$ cancel because $\psi$ solves the eikonal equation \r{o2} with the plus sign. Equate  the order $1$ terms, we must solve  
\be{trans}
\left(   \partial_t - \frac{\partial \lambda_1}{\partial\xi_j} \frac{\partial}{\partial x^j} \right)a_0 -\bigg( \i\lambda_0+ \sum_{|\alpha|=2}  \frac{\partial^\alpha_\xi\lambda_1}{\alpha!}\partial_x^\alpha\phi \bigg) a_0  =0,
\ee
where $\lambda=\lambda_1+\lambda_0+\dots$ is the expansion of $\lambda$ and they are evaluated at $\xi=\nabla_x\phi$. In our case, $\lambda_1= -c(x)|\xi|_{{\gzero}}$, therefore, $\partial \lambda_1/\partial \xi = -cg^{-1}\xi/|\xi|_{{\gzero}}$, which for $\xi=\nabla_x\phi$ yields  $\partial \lambda_1/\partial \xi =  -c{\gzero}^{-1}\nabla_x\phi/|\nabla_x\phi|_{{\gzero}} =c^2 {\gzero}^{-1}\nabla_x\phi/\phi_t$. Therefore, the vector field in \r{trans} is proportional to  the vector field $(\phi_t, c^2 {\gzero}^{-1}\nabla_x\phi)$ which is the Hamiltonian covector field of the wave equation \r{ac} on $\Sigma_+$ identified with a vector one, since the Laplacian there is the one associated with the metric $\tilde g:= c^{-2}{\gzero}$. As it is well known, this is also the geodesic vector field of $\tilde g$ in the tangent bundle. 
The potential-like term in \r{trans} involves $\lambda_0=-A_+$, see \r{4.5}. Now, the transport equation \r{trans} is a first order linear ODE along the bicharacteristics for the vector valued $a_0$ with a matrix valued zero  order potential-like term. Given initial conditions at $t=0$, it is solvable as long as $\phi$ is well defined.

The higher order transport equations for $a_1$, $a_2$, etc., are derived in a similar way. They are non-homogeneous, with the same left-hand side but on the right we have functions computed in the previous steps. 

We return to \r{o1} now and look for $u$ as a sum of four terms as indicated here, each one of the type we described. We can use the Cauchy data to derive initial conditions for the transport equations, see e.g., \cite{St-Encyclopedia}, to complete the construction. 

The integrals appearing in \r{o1}  are Fourier Integral Operators (FIOs) either with $t$ considered as a parameter, or as $t$ considered as one of the variables. In the former case, singularities of $(h_1,h_2)$ propagate along the zero bicharacteristics. More precisely,  for every $t$, 
\be{C1}
\WF(\mathbf{u}(t,\cdot)) = C_+(t)\circ\WF(\mathbf{h}) \cup C_-(t)\circ\WF(\mathbf{h}),
\ee
where $\mathbf{u}:=(u,u_t)$, $\mathbf{h}=(h_1,h_2)$ and 
\[
\begin{split}
C_+(t)(x,\xi) &= \left( \gamma_{x,\xi/|\xi|_{\tilde g}}(t),|\xi|_g  g\dot \gamma_{x,\xi/|\xi|_{\tilde g}}(t)  \right), \\
C_-(t)(x,\xi)& = \left( \gamma_{x,-\xi/|\xi|_{\tilde g}}(t), -|\xi|_{\tilde g} {\tilde g}\dot \gamma_{x,-\xi/|\xi|_{\tilde g}}(t)  \right) = C_+(-t)(x,\xi),
\end{split}
\]
and for $(x,\eta)\in T^*\R^3\setminus 0$, $\gamma_{x,\eta}$ is the geodesic issued from $x$ in direction $\tilde g^{-1}\eta$. 
 
On the other hand, considering $t$ as one of the variables, 
\be{C2}
\WF(\mathbf{u}) = C_+\circ\WF(\mathbf{h}) \cup C_-\circ\WF(\mathbf{h}),
\ee
where 
\[
\begin{split}
C_+(x,\xi) &= \left\{\left( t ,\gamma_{x,\xi/|\xi|_{\tilde g}}(t), -|\xi|_{\tilde g}, |\xi|_{\tilde g}{\tilde g}\dot \gamma_{x,\xi/|\xi|_{\tilde g}}(t)  \right),\; t\in\R\right\}, \\
C_-(x,\xi)& = \left\{\left( t,\gamma_{x,-\xi/|\xi|_{\tilde g}}(t), |\xi|_{\tilde g}, -|\xi|_{\tilde g} {\tilde g}\dot \gamma_{x,-\xi/|\xi|_{\tilde g}}(t)\right)\; t\in\R\right\}.
\end{split}
\]
 In the analysis below, we will consider $C_+$ only. 
 
The construction  above can be done  in some neighborhood of a fixed point $(0,x_0)$ in general. To extend it globally, we can localize it first for $\mathbf{h}$ with $\WF(\mathbf{h})$ in a conic neighborhood of some fixed $(x_0,\xi^0)\in T^*\R^3\setminus 0$. 
Then $u$ will be well defined near the geodesic issued from that point but in some neighborhood of $(0,x_0)$ in general. We can fix some $t=t_1$ at which $u$ is still defined,  take the Cauchy data there and use it to construct a new solution. Then we get an FIO which is a composition of the two local FIOs each one associated with a canonical diffeomorphism, then so is the composition. Then we can use a partition of unity to conclude that while the representation \r{o1} is local, the conclusions \r{C1} and \r{C2} are global. In fact, it is well known that both $\mathbf{h}\mapsto \mathbf{u}$ and  $\mathbf{h}\mapsto \mathbf{u}(t,\cdot)$ with $t$ fixed are global FIOs associated with the canonical relations in \r{C1} and \r{C2}. 
 
 In particular, if $\Gamma$ is a smooth hypersurface, and $\gamma_{x,\xi}(t)$ hits $\Gamma$ for the first time $t=t(x,\xi)$ transversely locally, then $\mathbf{h}\mapsto u|_\Gamma$ is an FIO again with a canonical relation as $C_+$ above but with $t=t(x,\xi)$ and $\dot\gamma$ replaced by its tangential projection $\eta':= \dot\gamma'$. Notice that $\tau=-|\xi|_{\tilde g}<0$ for $C_+$ and $\tau=|\xi|_{\tilde g}>0$ for $C_-$. Also, $|\tau|<|\eta'|_{\tilde g}$ with equality for tangent rays that we exclude; therefore, $\WF(u|_{\R\times \Gamma})$ is in the hyperbolic region, as defined below.

\subsection{The boundary value problem for the acoustic equation} \label{sec_Ac_BVP} 
Let $\Gamma$ be a smooth hypersurface near a fixed point $x_0$ given locally by $x^n=0$. We take $x=(x',x^n)$ to be local semigeodesic coordinates. We define $\Omega_\pm=\{\pm x^n>0\}$ to be the ``positive'' and the ``negative'' sides of $\Gamma$. At the beginning, we work in $\Omega_+$ only and omit the superscript or the subscript $+$ from the corresponding quantities. For all possible solutions $u$ (not restricted to incoming or outgoing ones) with singularities not tangent to $\Gamma$, we want to understand how the Dirichlet data $f:= u|_{\R\times \Gamma}$ and the Neumann data $h:= \partial_\nu u|_{\R\times \Gamma}$ are related. Once we have this, we can understand microlocally  the boundary value problems with either Dirichlet or Neumamn boundary conditions, or with Cauchy data. 

The analysis depends on where the wave front set of the Cauchy data is. 
Let $(f,h) \in \mathcal{E}'(\R\times \R^{n-1})$ be supported near some $(t_0,x')$. 
Then $T^*(\R\times \R^{n-1})\setminus 0$ has a natural decomposition into the \textit{hyperbolic region} $c^2|\xi'|_{{\gzero}}< \tau^2 $, the \textit{glancing one}  $\tau^2=c^2 |\xi'|_{{\gzero}}$, and  the \textit{elliptic one}  $c^2 |\xi'|_{{\gzero}}> \tau^2$. Each one has two disconnected components corresponding to $ \mp\tau>0$. We will recall the analysis in  the $\tau<0$ component in more detail and will point out the needed changes when $\tau>0$. Also, we will not analyze (a neighborhood of) the glancing region; for that, see, e.g., \cite{Taylor-book0} for a strictly convex boundary.  We are looking for a parametrix of the \textit{outgoing}  solution $u$ of \r{ac} with boundary data $f$, i.e., the solution with singularities propagating in the future only. Solutions with singularities propagating to the past only will be called \textit{incoming}.

\subsubsection{The outgoing and the incoming Neumann operators} 
If $u_\textrm{out}$ is the outgoing solution with boundary data $f$ with $\WF(f)$ in the hyperbolic region, we call the operator $\Lambda_\textrm{out}f=\partial_\nu u|_{\R\times \Gamma}$ the outgoing Neumann operator. Similarly we define the incoming Neumann operator by $\Lambda_\textrm{in}$. In those definitions, it is implicit that the solutions are defined in $\bar\Omega$ and $\nu$ is the unit normal exterior to it. i.e., $\partial_\nu=-\partial_{x^n}$. If we have $\Omega_\pm$ as above, we use the notation $\Lambda^\pm_\textrm{in}$, $\Lambda^\pm_\textrm{out}$ to denote the four Neumann operators with the convention that we preserve $\nu$ for $\Omega_0$, i.e., $\nu$ is interior for it. If the coefficients of the wave equation are smooth across $\Gamma$, we have $\Lambda_\textrm{out} ^+=  \Lambda_\textrm{in} ^-$,  $\Lambda_\textrm{in} ^+=  \Lambda_\textrm{out} ^-$ up to smoothing operators. In the transmission problem below however, this is not the case.

\subsubsection{Wave front set in  the hyperbolic region $c^2|\xi'|_{{\gzero}}<\tau^2$}  Assume that $\WF(f)$ is in the hyperbolic region with $\tau<0$.   We are looking for a representation of $u$ of the form 
\be{10c}
u = (2\pi)^{-n}\iint_{\R\times\R^{n-1}} e^{\i\phi(t,x,\tau,\xi')}  a(t,x,\tau,\xi') \hat f(\tau,\xi') \, \d\tau\, \d\xi', 
\ee
with a phase function  $\phi$ and an amplitude $a$. 
 
The phase function solves the eikonal equation in \r{o2} with the plus sign  but with a  boundary condition on the timelike boundary $x^n=0$ now
\be{eik0}
\partial_t \phi  +c(x)|\nabla_x\phi|_g=0, \quad \phi|_{x^n=0} = t\tau +x'\cdot\xi'.
\ee
The choice of the positive square root reflects the assumption $\tau<0$. In the hyperbolic region, there are two solutions depending on the choice of the sign of $\partial_{x^n}\phi$ at $x^n=0$. It is easy to see that what corresponds to outgoing solutions is the positive choice 
\be{4.11}
\partial_{x^n}\phi\big|_{x^n=0}= \sqrt{ c^{-2}\tau^2- |\xi'|_g}.
\ee
We solve \r{eik0} with this condition locally. To construct the amplitude, we solve the same transport equations \r{trans} as above but with initial condition $a=1$ for $x^n=0$, i.e., the principal part $a_0$ of $a$ is one there; and all others vanish.

The case $\tau>0$ is similar: we seek the solution in a similar way but the sign in \r{eik0} is negative. This does not change the construction.

Incoming solutions are constructed similarly. We choose the negative square root in \r{4.11}. in particular we get that the outgoing and the incoming Neumann operators are \PDO s of order one with principal symbols equal to $\mp \i$ multiplied by \r{4.11}, see also Proposition~\ref{pr_N} below.

\subsubsection{Wave front set in  the elliptic region $c^2|\xi'|> \tau^2$. Evanescent waves}\label{sec_Ac_evan} 
We proceed formally in the same way but the problem here is that the eikonal equation has no real valued solution because the expression under the square root in \r{4.11} is negative. It may not even have  a complex valued solution. 
This is a well known case of an evanescent mode described by a complex valued phase function (and amplitude).  We follow \cite{Gerard}, see also \cite[VIII.4]{Taylor-book0}.  Since the construction in \cite{Gerard} is done for the Helmholtz equation with a large parameter and in  \cite[VIII.4]{Taylor-book0} it is done for an elliptic boundary value problem, respectively, we need to do them in our hyperbolic case as well,  even though the construction is essentially the same. 
We assume that $(t,x,\tau,\xi')$ belong to a conically compact neighborhood, contained in the elliptic region, of a fixed  point there. 
Plugging the ansatz in the elasticity equation, we use the ``fundamental lemma'' for complex phase functions in \cite[X.4]{Treves2} to get an asymptotic expansion which formally look the same as in the hyperbolic case. 
We are looking for a solution of the eikonal equation \r{11} for $\phi$ up to an error $O(|x^n|^\infty)$ at $x^n=0$ as a formal infinite expansion of the form 
\[
\phi = t\tau+ x'\cdot\xi'+ x^n\psi_1(t,x',\tau,\xi')+(x^n)^2 \psi_2(t,x',\tau, \xi')+\dots,
\]
where $\psi_j$ are symbols of order $1$. We denote this class by $\tilde S^1$, and by replacing the order $1$ by some $m$,  we denote by $\tilde S^m$ the corresponding class. 
To avoid exponentially large modes, we require $\Im \phi \ge0$. To construct the formal series, we first write  the eikonal equation \r{4.11} in the form 
\be{6.9a}
\partial_{x^n}\phi = \i\sqrt{|\nabla_x'\phi|_g^2 - (\partial_t \phi)^2}
\ee
(note that there are no incoming/outgoing choices here)  and then differentiate it w.r.t.\ $x^n$ at $x^n=0$. If such a solution exists, the error term would not affect those derivatives. We have
\be{6.9aa}
\psi_1 = \i\sqrt{ |\xi'|_g^2-c^{-2} \tau^2  }.
\ee
To find the higher order derivatives, we write \r{6.9a} in the form
\[
\partial_{x^n}\phi = F(x,\partial_{t,x}\phi);
\]
with $F(x,\eta)$ homogeneous in $\eta$ of order one. Then
\[
\partial_{x^n}^{k+1}\phi = \sum_{|\beta|+k_0+k_1+\dots+k_{|\beta|}=k} \partial_{x^n}^{k_0} \partial_\eta^\beta  F(x,\partial_{t,x} \phi) \partial_{x^n}^{1+k_1}\phi_{t,x} \dots \partial_{x^n}^{1+k_{|\beta|}}\phi_{t,x} .
\]
Since    $\partial_{x^n} \phi$   is a symbol of order one, we prove the claim. Note also that $\Im\phi\ge x^n(|\tau|+|\xi|)/C$.

The next step is to solve the transport equations. Since they have complex coefficients, they may not be solvable exactly and we solve them up to an $O(|x^n|^\infty)$ error as well. The rest is as in  \cite{Gerard} and \cite[VIII.4]{Taylor-book0}.

\begin{proposition}\label{pr_N}
In the hyperbolic region, $\Lambda_\text{\rm out}$ and $\Lambda_\text{\rm in}$ are \PDO s of order one with principal symbols
\be{pr_N1}
\sigma_p(\Lambda_\text{\rm out}) = -\i \sqrt{ c^{-2}\tau^2- |\xi'|_g}, \quad \sigma_p(\Lambda_\text{\rm in}) = \i\sqrt{ c^{-2}\tau^2- |\xi'|_g}.
\ee
In the elliptic one, they are \PDO s of order one again with principal symbols
\be{pr_N2}
\sigma_p(\Lambda_\text{\rm out}) =  \sigma_p(\Lambda_\text{\rm in}) = \sqrt{|\xi'|_g-  c^{-2}\tau^2}. 
\ee
\end{proposition}
We recall that $\partial_\nu= - \partial_{x^n}$ in the coordinates we used to compute the principal symbols. The expressions we got are invariant however.  In both  cases, the DN maps are elliptic. As shown in \cite{Taylor-book0}, they are elliptic even in the glancing region but they belong to a different class of \PDO s. The principal symbols of the Neumann operators on the negative side $\Omega_-$ are similar but with opposite signs. 

\subsubsection{The boundary value problem with Dirichlet data} \label{sec_Ac_BVP_D} The problem of constructing the outgoing solution $u_\text{out}$ with Dirichlet data on $\R\times \Gamma$ was solved above when $\WF(f)$ is either in the hyperbolic of the elliptic region. Similarly, we construct $u_\text{in}$. Notice that in the elliptic region, the construction is the same for both. In particular, we proved Proposition~\ref{pr_N} by taking the normal derivatives of those solutions. 

Next, we can construct a reflected wave. Assume we have an incoming solution $u_\text{in}$ with singularities hitting $\Gamma$ transversely.  We want to construct a solution $u$ equal to $u_\text{in}$ for $t\ll0$ satisfying $u=0$ on the boundary.   Then $f:=u|_{\R\times \Gamma}$ has a wave front set in the hyperbolic region only. We construct the reflected wave $u_R$ as the outgoing solution with Dirichlet data $-f$. Then $u=u_\text{in}+u_R$ is the solution we seek.

\subsubsection{The boundary value problem with Neumann data} \label{sec_Ac_BVP_N} 
Consider the outgoing solution $u_\text{out}$ with boundary data $\partial_\nu u=h$ on $\R\times \Gamma$. We reduce it to the Dirichlet problem above by inverting the DN map in $\Lambda_\text{out}f=h$. Since the latter is elliptic in the two regions we work in, this can be done microlocally. Then we solve a Dirichlet problem. We do the same for the incoming solution. 

If we want to construct a reflected wave so that the solution $u$ satisfies $\partial_\nu u=0$, we need to solve $N_\text{out}f = -\partial_\nu u_\text{in}|_{\R\times \Gamma}$ which is possible since $N_\text{out}$ is elliptic. Having $f$, then we construct the outgoing solution with that Dirichlet data.

\subsubsection{The boundary value problem with Cauchy data} \label{sec_Ac_BVP_C}\ 
We are looking for a microlocal solution $u$ of the acoustic equation \r{ac} satisfying $u=f$ and $\partial_\nu u=h$ on $\R\times \Gamma$ with given $f$ and $h$ having wave front sets in the hyperbolic region first. The global Cauchy problem is over-determined because the singularities can hit the boundary again and therefore the Cauchy data have a structure (consisting of pairs in the graph of the lens relation); therefore prescribing them arbitrarily is not possible. On the other hand, one can construct a microlocal solution locally, when the wave front sets of $f$ and $h$ are localized in small conic sets excluding tangential directions, until  the singularities hit the boundary again. We are looking for $u$ as a sum of two solutions $u=u_\text{in}+u_\text{out}$, one incoming and the other one outgoing.  To determine the boundary values of the two solutions and to reduce the problem to section~\ref{sec_Ac_BVP_D}, we need to solve
\be{7.1}
u_\text{in} +  u_\text{out}=f, \quad 
\Lambda_\text{in} u_\text{in} + \Lambda_\text{out} u_\text{out} =h,
\ee
where $u_\text{in}$ and $u_\text{out}$ are the boundary values of those solutions. 

Let $\WF(f,h)$ be in the hyperbolic region first. 
Then on principal symbol level, the leading amplitudes  solve
\[\
a_\text{in} + a_\text{out}=\hat f, \quad -\i\xi_3(a_\text{in}  - a_\text{out} )=\hat h\quad \text{on $x^3=0$},
\]
where $\xi_3$ is defined by \r{4.11}. This in an elliptic system. This shows that the matrix valued operator in \r{7.1} is elliptic (if we reduce the order of the second equation to $0$ by applying an elliptic \PDO\ of order $-1$). 
Therefore, the Cauchy data determine uniquely a decomposition into an incoming and an outgoing solution, locally. This reduces the problem to the one we solved in section~\ref{sec_Ac_BVP}. 

 If $\WF(f,h)$ is in the elliptic region, there is only one parametrix, no incoming or outgoing ones. The corresponding DN map $\Lambda$ is an elliptic \PDO \ of order one with principal symbol \r{pr_N2}. Then for $(f,h)$ to be Cauchy data of an actual solution (up to smooth functions) it is needed that it belongs to the range of $(\Id,\Lambda)$ (up to smooth functions). This makes this problem over-determined. If $h=\Lambda f$, a microlocal solution exists, as we showed above. It propagates no singularities away from $\Gamma$, and it does not propagate singularities along $\Gamma$ either (unlike the Rayleigh waves in elasticity which propagate along $\Gamma$).

\subsection{The transmission problem}\label{sec_Ac_RT}
  We recall the setup in section~\ref{sec_Ac_BVP}. 
We work locally in a small neighborhood of a point on $\Gamma$ and call one of its sides, $\Omega_-$ negative, the other one, $\Omega_+$, positive. For the speed $c$, we have $c=c_{-}$ in $\Omega_-$, and $c = c_{+}$ in $\Omega_-$, where $c_{-}, c_{+}$ are smooth up to $\Gamma$ and $c_-\not=c_+$ pointwise. We impose the transmission conditions
\be{trA}
[u]=[\partial_\nu u]=0\quad\text{on $\Gamma$},
\ee
where $\nu$ is the normal derivative. 
 Let $(x',x^3)$ be semi-geodesic coordinates near $\Gamma$ so that $\pm x^3>0$ in $\Omega_\pm$.

 Let $u_I$ be an incident solution of the acoustic equation \r{ac} with speed $c$ and background metric ${\gzero}$ with a wave front set localized near a small conic neighborhood of some covector (at some time) approaching $\Gamma$ from the positive  side. $\Omega_+$ As mentioned above, we consider singularities $(x,\xi)$ which move in the direction of $\xi$ only, i.e, associated with $\phi_+$ in \r{o1}, as we did in section~\ref{sec_GO}. Then on $\WF(u_I)$, with $t$ considered as a variable, we have $\tau<0$. 
 Extend the speed $c$ form the negative to the positive side in a smooth way (recall that $c$ jumps across $\Gamma$) and extend $u_I$ smoothly across $\Gamma$ as a solution with that speed.  Set 
\be{14h}
f:= u_I|_{\R\times \Gamma}.
\ee
Let $(x_0,\xi_0)$ with $x_0\in \Gamma$ be one of the singularities of $u_I$. We assume that $\xi_0$ is a unit covector w.r.t.\ $c_{+}^{-2}{\gzero}$. We have that $\WF(f)$ is in the hyperbolic region  $c_+|\xi'|<-\tau$ in $\Omega_+$. We are looking for a parametrix $u$ near $x_0$ of the form
\be{14u}
u = u_I + u_R+u_T,
\ee
where $u_I$ is incoming and restricted to $\bar\Omega_+$;  $u_R$ is the reflected outgoing solution supported in $\bar\Omega_+$, and $u_T$ is the transmitted outgoing one or an evanescent mode, supported in $\bar\Omega_-$. It is enough to find the boundary values of those functions. 

\subsubsection{The hyperbolic-hyperbolic case} 
Assume that $\WF(f)$ is in the hyperbolic region in $\Omega_-$ as well, i.e., $c_{-}^2|\xi'|^2<\tau^2$ on $\WF(f)$. If $c_{-}<c_{+}$ at $x_0$ (transmission from a fast to a slow region), that condition is satisfied regardless of $\xi_0'$.  If $c_{-}>c_{+}$(transmission from a slow to a fast region), existence of a transmitted ray depends on $\xi_0'$.  
Let $\theta_+$ be the angle which an incoming ray makes with the normal, then the reflected  angle will be the same and the angle  $\theta_-$ of the transmitted ray, see Figure~\ref{pic1ac}, is related to $\theta_+$ by  Snell's law
\be{Snella}
\frac{\sin\theta_+}{\sin\theta_-} = \frac{c_+}{c_-},
\ee
which  follows directly from \r{eik0} with $c=c_-$ and $c=c_+$ there, see also \cite{SU-thermo_brain}.  This relation shows that a transmitted ray will exist only if $\theta_+$ does not  exceed the critical angle 
\be{theta_cr}
\theta_\textrm{cr}=\arcsin(c_+/c_-).
\ee

\begin{figure}[!ht]
\includegraphics[page=7,scale=1]{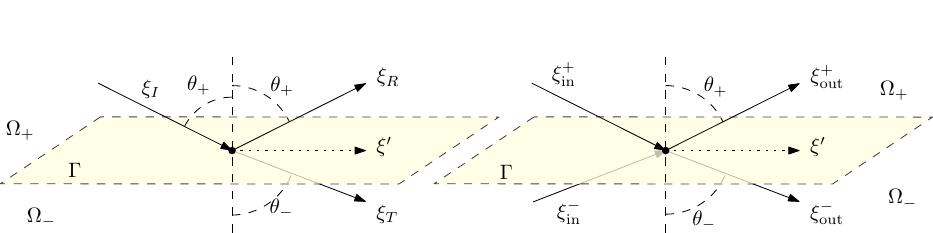}
\caption{Reflected and transmitted acoustic waves with an incoming ray from the top (left) and with incoming rays from both sides (right)
}\label{pic1ac}
\end{figure}

The transmission conditions \r{trA} are equivalent to 
\be{acTC1}
\begin{split}
u_I + u_R & =  u_T,\\
N_\textrm{in}^+u_I  + N_\textrm{out}^+u_R & =   N_\textrm{out}^-u_T.
\end{split}
\ee

Assume now that we want to satisfy transmission conditions requiring continuity of $u$ and its normal derivative across the boundary. Then we get the following linear  system for the leading terms $a_T^{(0)}$ and $a_R^{(0)}$ of the amplitudes $a_T$ and $a_R$:
\be{15}
\begin{array}{rll}\medskip 
a_T^{(0)}-a_R^{(0)} &=a_I^{(0)} & \text{for $x^n=0$},\\ \displaystyle 
-\xi_n^- a_T^{(0)} -  \xi_n^+  a_R^{(0)} &=\displaystyle -\xi_n^+ a_I^{(0)} &\text{for $x^n=0$}, 
\end{array}
\ee
where 
\be{15b}
\xi_n^\pm  = \sqrt{c_{\pm}^{-2} \tau^2-|\xi'|_g^2},  \quad \text{for $x^n=0$}.
\ee
In particular, this shows that the determinant of \r{15} is negative, and therefore, the system is solvable, i.e., elliptic after reducing the order of the second equation to zero.
Since the system \r{acTC1} is elliptic, it can be solved up to infinite order, i.e., we can find the all terms $a_{R,T}^{(k)}$ at $x^n=0$. The solutions serve as initial conditions for the transport equations of   the corresponding modes. 

Multiplying the first by the conjugate of the second equation, we get
\[
\xi_n^-|a^{(0)}_T|^2+ \xi_n^+|a^{(0)}_R|^2 = \xi_n^+|a^{(0)}_I|^2,
\]
which can be considered (and justified) as preservation of the energy  across $\Gamma$.

\subsubsection{Total internal reflection}\label{sec_ac_FIR}
 Assume now that $\WF(f)$ is in the elliptic region for $c_-$. This happens when $\theta_+> \theta_\textrm{cr} $. In that case, there will be no transmitted singularity. Indeed, we are looking for an evanescent mode in $\Omega_-$. Then $N_\text{out}^-$ in \r{acTC1} is in the elliptic region. The analog of \r{15} then is
\be{ac_fir1}
\begin{pmatrix} 1&-1\\ -\xi_n^-&-\xi_n^+ \end{pmatrix} \begin{pmatrix}  a_T^{(0)}\\  a^{(0)}_R \end{pmatrix}
= a_I^{(0)}\begin{pmatrix} 1\\ -\xi_n^+\end{pmatrix}  
\ee
where $\xi_n^-=\i |\xi_n^-|$ is pure imaginary and given by \r{pr_N2} times $\i$. Equivalently,
\be{ac_fir2}
a_I^{(0)}+ a_R^{(0)}= a_T^{(0)}, \quad  \xi_n^+\left( a_I^{(0)}- a_R^{(0)}\right)= \i |\xi_n^-|   a_T^{(0)}.
\ee
Take the real part of the first equation multiplied by the conjugate of the second one to get 
\be{ac_fir3}
\big|a_R^{(0)}\big|^2 = \big|a_I^{(0)}\big|^2.
\ee
In other words, on principal level, the whole energy is reflected and nothing is transmitted. We could have obtained this  directly by solving \r{ac_fir1}, of course.

\subsubsection{Incoming waves from both sides of $\Gamma$. } 
A more general setup is to assume incoming waves from each side, see Figure~\ref{pic1ac}, right. We do not need to assume hyperbolic ones; they could be evanescent. In fact, this is an analogue of the Cauchy data case in the boundary value problem, see section~\ref{sec_Ac_BVP_C}. The point of view we adopt and will keep in the elastic case, is to classify the cases by the wave front set of the Cauchy data on the boundary. 

We are interested in two questions: (i) well posedness of the transmission problem: given all incoming waves, is the problem well posed for the outgoing ones; and (ii) given all waves on one side of $\Gamma$, can we solve for all waves on the other one? We show that (i) is true as it can be expected (and well known). The answer to (ii) is not always affirmative; and when it is; this means that we can control the configuration on one side from the other one; in particular we can kill either the  incoming  or the outgoing wave on that side.

\textbf{The hyperbolic-hyperbolic case.} We assume now that the Cauchy data $(f,h)$ (the same on both sides by the transmission conditions) has a wave front set in the hyperbolic region on each side of $\Gamma$. Then on each side, we have two solutions: one incoming and one outgoing. 
Let $u^+_{\textrm{in}}$ and $u^-_{\textrm{in}}$ be the two incoming solutions from the positive and from the negative side, respectively, and let  ${u}^+_{\textrm{out}} $, ${u}^-_{\textrm{out}}$ be the two outgoing ones. A usual, we assume no tangent rays. Then the transmission conditions are given by 
\be{acTC}
\begin{split}
u_\textrm{in}^++ u_\textrm{out}^+ & = u_\textrm{in}^-+ u_\textrm{out}^-,\\
N_\textrm{in}^+u_\textrm{in}^++ N_\textrm{out}^+u_\textrm{out}^+ & = N_\textrm{in}^- u_\textrm{in}^-+ N_\textrm{out}^-u_\textrm{out}^-.
\end{split}
\ee
This is a generalization of \r{acTC1} with one more wave added. 
If the corresponding principal amplitudes are $ a^+_{\textrm{in}}$, $a^-_{\textrm{in}}$,  $a^+_{\textrm{out}}$, $a^-_{\textrm{out}}$, we get
\be{ac20}
\begin{pmatrix} 1&1\\ -\xi_n^+&\xi_n^+ \end{pmatrix} \begin{pmatrix}  a_\textrm{in}^+\\  a^+_\textrm{out} \end{pmatrix}
= \begin{pmatrix} 1&1\\ \xi_n^-&-\xi_n^- \end{pmatrix} \begin{pmatrix}   a^-_\textrm{in}\\   a^-_\textrm{out}  \end{pmatrix}
\ee
Clearly, each matrix is elliptic. This implies that we have control from each side:  given any choice of two amplitudes on one side, say  $\Omega_-$, one gets an elliptic problem for finding the amplitudes on the other one, in this case $\Omega_+$.  

We also get ellipticity for solving for the outgoing/incoming waves given the incoming/outgoing ones, i.e., the transmission problem is well posed. This also follows from energy conservation. Indeed, multiplying the first by the conjugate of the second equation, and then taking the real part above yields
\be{ac_en1}
\xi_n^+ \left( |a^+_{\textrm{out}}|^2 - |a^+_{\textrm{in}}|^2  \right) + \xi_n^- \left( |a^-_{\textrm{out}}|^2 - |a^-_{\textrm{in}}|^2  \right)=0. 
\ee
This  energy preservation  across the boundary implying in particular that if all incoming waves vanish, then so do the outgoing ones; i.e., that problem is elliptic. 

\textbf{The hyperbolic-elliptic case.} We assume now that the Cauchy data $(f,h)$ (the same on both sides by the transmission conditions) has a wave front set in the hyperbolic region w.r.t.\ $c_+$ and in the elliptic one for $c_-$. Then in $\Omega_+$ we have two solutions: one incoming and one outgoing but in $\Omega_-$ there is only one (evanescent) solution. This case is analyzed in section~\ref{sec_ac_FIR} with $u_\text{in}^+=u_I$, $u_\text{out}^+=u_R$, $u^-$ (no incoming or outgoing ones) corresponding to $u_T$ there. We found out there that the incoming wave (or the outgoing one) determines uniquely the outgoing (respectively, the incoming) one and the evanescent one $u_-$. On the other hand, we cannot control $u_\text{out}^+$ and  $u_\text{in}^+$ by choosing appropriately the evanescent mode $u^-=u_T$ appropriately; in fact $u_\text{in}^+$ alone determines the whole configuration already. 

A slightly different point of view into this case is that we cannot have arbitrary (up to smooth functions) Cauchy data on $\Gamma$ in the hyperbolic region for $\Omega_+$, since that data falls in the elliptic region on the negative side, and then it has to be in the graph of the Neumann operator $\Lambda_-$. On other hand, if that data satisfy that compatibility condition, the solution in $\Omega_+$ consists of an incoming and a reflected wave. This is in contrast to the hyperbolic-hyperbolic case, where we can cancel one of the waves on the top, for example.

\textbf{The elliptic-elliptic case.} We assume now that the Cauchy data $(f,h)$  has a wave front set in the elliptic  region w.r.t.\ both $c_+$ and $c_-$. It is interesting to see if we can have evanescent modes on both sides but still a non-trivial wave front set on $\Gamma$. We would need $(|\xi'|_g^2-c_+^{-2} \tau^2 )^{1/2}=- (|\xi'|_g^2-c_-^{-2} \tau^2 )^{1/2}$ which cannot happen. Therefore, there are no Rayleigh or Stoneley kind of waves in the acoustic case. 

\subsection{Justification of the parametrix} \label{sec_just_ac}
In each particular construction up to section~\ref{sec_Ac_BVP_C}, we constructed a parametrix satisfying the equation and the corresponding initial/boundary conditions up to a smooth error. Then the difference of the parametrix and the true solution satisfies all those conditions up to smooth errors. Standard  hyperbolic estimates imply that the difference is  smooth. In section~\ref{sec_Ac_BVP_C}, the Cauchy problem on a timelike boundary needs to be solved microlocally only and it is a tool to handle the transmission one. The justification of the parametrix for the latter can be done with the aid of \cite{Hansen84, Williams-transmission}, guaranteeing smooth solutions if the transmission conditions \r{tr} hold up to a smooth error only. 

\section{Geometric optics for the elastic wave equation} \label{sec_GOel} 

We study the Cauchy problem at $t=0$ and propagation of singularities in the elastic case. 
We present the geometric optics construction for the elastic wave equation in an open set first, where the coefficients are smooth. Such a construction is well known for systems with characteristics of constant multiplicities, see, e.g., \cite{Taylor-book0, Treves} and \cite{Dencker_polar}.  Our goal is to make the elastic case more explicit and to do a complete mode separation which we will use eventually near a boundary, see Proposition~\ref{pr1} below. 
The elastic case has been studied form microlocal point of view in \cite{Yamamoto_elastic_89, Rachele_2000, Rachele00, Rachele03, HansenUhlmann03, Brytik_11, SUV_elastic}.

Consider the elastic wave equation 
\be{el}
\begin{split}
u_{tt}-Eu&=0,\\
(u,u_t)|_{t=0}&=(h_1,h_2)
\end{split}
\ee
with Cauchy data $\mathbf{h}:=(h_1,h_2)$ at $t=0$. We want to solve it microlocally for $t$ in some interval and $x$ in an open set. The operator $E$ is associated with a Riemannian metric $g$ as in section~\ref{sec_m}. 
If $\lambda$, $\mu$ and $\rho$ are constant and $g$ Euclidean, one can use Fourier multipliers. In that case, 
let $\Pi_p=\Pi_p(D)$ be the projection to the p-modes, i.e., $\Pi_p$ is the Fourier multiplier $\hat u\mapsto (\xi/|\xi|)[(\xi/|\xi|)\cdot \hat u] $ and let $\Pi_s=\Id-\Pi_p$. It is easy to see that 
$\Pi_s$ is the Fourier multiplier $\hat u\mapsto -(\xi/|\xi|)\times (\xi/|\xi|)\ \times \hat u$. Also, we may regard $h = \Pi_ph+\Pi_s h$ as the potential/solenoidal (or the Hodge) decomposition of the 1-form $h$, see, e.g., \cite{Sh-book}. Then , $ {E} = c_p^2\Delta\Pi_p+ c_s^2\Delta \Pi_s$. We have a complete decoupling of the system into P and S waves.

In the variable coefficient case, we will do this up to smoothing operators. 
We recall  the construction in \cite{Taylor-book0}, which provides another proof of the propagation of singularities in this case. 
The principal symbol $\sigma_p(-E)$ of $-E$ has eigenvalues of constant multiplicities. Near every $(x_0,\xi_0)\in T^*\bar\Omega\setminus 0$, one can decouple the full symbol  $\sigma(-E)$ fully up to symbols of order $-\infty$. Namely, there exist  an elliptic matrix valued \PDO\  $U$  of order $0$  microlocally defined near $(x_0,\xi_0)$, so that 
\be{VEU}
U^{-1}EU= \begin{pmatrix}  c^2_s\Delta_g + A_s&0\\0&c_p^2\Delta_g +A_p\end{pmatrix}
\ee
modulo $S^{-\infty}$ near $(x_0,\xi_0)$,  where the matrix is in block form; with an $1\times 1$ block on the lower right and a $2\times2$ one on the upper left ($c_s^2\Delta_g+A_s$ is actually $c_s^2\Delta_g\,I_2+A_s$ with $I_2$ being the identity in two dimensions). Moreover, $A_s$ and $A_p$ are \PDO s of order one. 
In other words, the top non-zero block is scalar and the lower non-zero one is principally scalar. We recall this construction briefly. We seek $U$ as a classical \PDO\ with a principal symbol $U_0$ which diagonalizes $E$; there are many microlocal choices, and we fix one of them. 
Then 
\be{VEU2}
U_0^{-1}EU_0= \begin{pmatrix} -c_s^2\Delta_g I_2&0\\0&-c_p^2\Delta_g \end{pmatrix} + R_1,
\ee
where $R_1$ is of order one. Then we correct $U_0$ by replacing it with $U_0(I +K_1)$ with some \PDO\ $K$ of order $-1$, i.e., we apply $I+K_1$ to the right and $(I+K_1)^{-1}= I-K_1+\dots$ to the left to get 
\be{VEU3}  
(I-K_1)U_0^{-1}EU_0(I+K_1)= (I-K_1)\begin{pmatrix} -c_s^2\Delta_g I_2&0\\0&-c_p^2\Delta_g \end{pmatrix}(I+K) + R_1,\quad \text{mod\ $\Psi^0$},
\ee
where we used the fact that $(I-K_1)R_1(I+K_1)=R_1$ mod $\Psi^0$. Let us denote the matrix operator there by $G$. To kill the off diagonal terms on the right up to zeroth order, we need to do that for $GK-KG+R$. Note that $G$ and $K_1$ do not commute up to a lower order because they are matrix valued \PDO s. We look for $K$ in block form with zero diagonal entries and off-zero ones $K_{12}$ (an $1\times2$ vector) and $K_{21}$ (a $2\times1$ vector). If we represent $R_1$ in a block form as well, we reduce the problem to solving
\[
\begin{split}
K_{12}(-c_s^2\Delta_g)  - (-c_p^2\Delta_g)K_{12}&=-R_{12},\\
K_{21}(-c_p^2\Delta_g) - (-c_s^2\Delta_g)  K_{21}&=-R_{21}
\end{split}
\]
modulo $\Psi^0$. 
The solvability of this system on a principal symbol level follows by the general lemma in \cite[IX.1]{Taylor-book0} because $c_s\not=c_p$ but in this particular case, it is straightforward. Note that the principal symbols of $K_{12}$ and $K_{21}$ represent the coupling of the P and the S waves on a sub-principal symbol level, see also \cite{Brytik_11}. 

We apply $I-K_2$ to the left and $I+K_2$ to the right to kill the off diagonal terms of $(I+K_1)^{-1}G(I+K_1)$, etc. In fact, $U$ can be chosen to be unitary in microlocal sense \cite{MR1777025}. In our case however, we prefer $U$ to be of order one. 

From now on, we will do all principal symbol computation at a fixed point where $g$ is transformed to an Euclidean one (via the exponential map, for example) to simplify the notation. Then we will interpret the final result in invariant sense. 

The principal symbol, of $U$, at that fixed point,   will be chosen to be 
\be{U}
\sigma_p(U) = \begin{pmatrix}    0&-\xi_3&\xi_1\\  \xi_3&0&\xi_2\\ -\xi_2 & \xi_1&\xi_3\end{pmatrix}
\ee
when $\xi_3\not=0$. The third column is the eigenvector $\xi$ associated with $c_p^2$, while the first and the second ones are a basis of the eigenspace of $\sigma_p(-E)$ associated with $\sigma_s$; and that basis is (micro) local only. In fact, a global one does not exist since those vectors are characterized as being conormal to $\xi$. In this particular case, we chose $\xi\times e_1$ and $\xi\times e_2$ with $e_1=(1,0,0)$, etc. 

Recall that the principal symbol  computations so far are at a single point where $g$ is Euclidean. To extend it to all points,  an invariant way to choose $\sigma_p(U)$ is to replace the first and the second column  there by $\xi\times e_1$ and $\xi\times e_2$ with $e_{1,2}$ considered as covectors, and the cross product as in \r{cross}. In other words, the first  two columns in \r{U} are considered as vectors, then converted to covectors by the metric and multiplied by $(\det g)^{-1/2}$. Then we still get \r{5.12} but in $u^s$ we have curl in terms of the metric, see \r{curl}. 

It then follows that microlocally, the elasticity system can be written as $(\partial_t^2 - U^{-1}EU) w=0$ for 
\be{uv}
w= (w^s,w^p)=  U^{-1}u,
\ee
where $w^s=(w_1^s,w_2^s)$ and $w^p$ is scalar. This system decouples into the wave equations 
\be{SP-dec}
\begin{split} 
(\partial_t^2-c_s^2\Delta_g -A_p)w^s-R_sw&=0,\\
(\partial_t^2-c_p^2\Delta_g -A_s )w^p-R_pw&=0, 
\end{split}
\ee
 with $A_{p,s}$ of order one, $R_{p,s}$ smoothing; the first one is  a  $2\times 2$ system     and the second one is scalar. The first one has $\Sigma_s$ as a characteristic manifold, while the second one has $\Sigma_p$. 
Even though $U$ depends on the microlocal neighborhoods of the characteristic varieties $\Sigma_{s,p}$ we work in, the wave front sets of $U^{-1}f$, in those neighborhoods, we can apply the propagation of singularities results, or directly the microlocal geometric optics construction used below. Then we conclude that singularities in those neighborhoods propagate along the zero bicharacteristics of $\tau^2-c_s^2|\xi|^2$ and $\tau^2-c_p^2|\xi|^2$, respectively (which, of course, is well known). This implies a global result, as well. 

For $u= Uw$ we get
\be{5.11}
 u= u^s +u^p, \quad   u^s:=   U( w^s_1,w^s_2,0), \quad   u^p:=   U(0,0,w^p), 
\ee
where $u^s$ and $u^p$ have wave front sets in $\Sigma_s$ and $\Sigma_p$, respectively. We call such solutions microlocal S and P waves. We have 
\be{5.12}
  u^p =    (D   + V_p)w^p, \quad u^s =   (\det g)^{-1/2} g(-D_3 w_2^s, D_3 w_1^s,  -D_2 w_1^s+D_1 w_2^s) + V_sw^s,
\ee
where $V_p$ and $V_s$ are of order zero  and are formed by the lower order entries of $U$. 
Here $u^s$ can also be written as $u^s = D\times (w_1^s,w_2^s,0)+V_sw^s$.

Therefore, we proved the following. 

\begin{proposition}[mode separation]\label{pr1} 
Let $u$ be a solution of the elastic wave equation in the metric setting in some open set in $\R\times\R^3$. Let $u^p$ and $u^s$ be $u$ microlocalized near $\Sigma_p$ and $\Sigma_s$, respectively. 
Then, microlocally,  in any conic subset where $\xi_3\not=0$,  there exist  a scalar function  $w^p$ and a  vector valued function $w^s=(w_1^s,w_2^s)$ solving \r{SP-dec} so that $u=u_p+u_s$,  where  
\be{5.13}
u^p =    (D  + V_p)w^p, \quad u^s = D\times (w_1^s,w_2^s,0 )+V_sw^s
\ee
with $V_p$ and $V_s$ \PDO s of order zero  and the curl in  $D\times$ is in Riemannian sense. 
\end{proposition}

The assumption $\xi_3\not=0$ does not restrict us. We can always rename the variables or rotate the coordinate system.  On the other hand, the proposition does not provide a global mode separation. We are going to use it with $x^3$ being the distance to the boundary. Note also that $u$ and $w^p$, $w^s$ are related by \r{uv}.  

In the geophysics literature, $w^p$ and $w^s$ such that $u^s=\nabla\times w^s$ (in our case,  $w^s=(w_1^s, w_2^s,0)$) are called potentials. We have some freedom to choose $w^s$ so that \r{5.13} hold: adding an exact form to $(w_1^s,w_2^s,0 )$ would not change the principal part of  $u^s$ at least. One possible gauge to get unique $w^s$  is to take one of the components, in some coordinate system, to be zero. We have $w_3^s=0$ in \r{5.13}. The analysis  however must be restricted microlocally to $\xi_3\not=0$.
In what follows, $x^3$ will be the normal coordinate to the boundary. Another choice is to require $w^s$  to be solenoidal, i.e., divergence free.  

This proposition is a generalization of the well known representation of the solution of the isotropic constant coefficient elastic equation into potentials $u=\nabla w^p +\nabla\times w^s$ solving \r{SP-dec} with the operators $A_{p,s}$ and $R_{p,s}$ there vanishing. To guarantee uniqueness, it is often assumed that $w^p=(-\Delta)^{-1}\nabla\cdot u$, $w^s= -(-\Delta)^{-1}\nabla\times u$. We can prove a version of this in the variable coefficient case as well.

\section{The boundary value problem for the elastic system. Dirichlet boundary conditions} \label{sec_el_bvp}
Consider the elastic wave equation $u_{tt}-{E}u=0$ with boundary data $u=f$ on $\R\times\bo$. Assume that $f=0$ for $t\ll0$ and we are looking for the outgoing solution, i.e., the one which vanishes for $t\ll0$. We also introduce the notion of a microlocally outgoing solution along a single bicharacteristic requiring singularities of such a solution to propagate  to the future. We define similarly incoming solutions by reversing time.  Note that an outgoing solution does not need to consist of microlocally outgoing ones only since some incoming ones may be canceled at interfaces by outgoing ones.  
We will construct a parametrix of those solutions using the analysis in section~\r{sec_Ac_BVP}. Moreover, we study the Cauchy data problem as well. We will use the analysis in the acoustic case essentially.


We work in semigeodesic coordinates $x=(x',x^3)$, with $x^3>0$ in $\Omega$. We denote the dual variables by $(\xi',\xi_3)$. The Euclidean metric then takes the form $g$ in those coordinates with $g_{\alpha 3}=\delta_{\alpha 3}$ for $1\le\alpha\le 3$.  
The analysis however works if we start with an arbitrary metric $g$ in $\R^n$, not just with the Euclidean one. Norms and inner products below are always in the metric $g$ or $g^{-1}$ (for covectors). 

The phase space on the cylindrical  boundary $\R\times \bo $ can be naturally split into the following regions (recall that $c_s<c_p$):

\begin{description}
\item [Hyperbolic region] $c_p|\xi'|_g<|\tau|$. Then $c_s|\xi'|_g<|\tau|$ as well, so it is hyperbolic for both speeds. 
\item [P-glancing region] $c_p|\xi'|_g=|\tau|$. It is glancing for $c_p$ and hyperbolic for $c_s$.  
\item [Mixed region] $c_s|\xi'|_g< |\tau|<c_p|\xi'|_g$. It is elliptic for $c_p$ but hyperbolic for $c_s$. 
\item [S-glancing region] $c_s|\xi'|_g=|\tau|$. It is glancing for $c_s$ and elliptic  for $c_p$.  
\item [Elliptic region] $|\tau|< c_s|\xi'|_g$. Then $|\tau|< c_p|\xi'|_g$, as well, so it is elliptic for both speeds. 
\end{description}

We will not analyze wave fronts in the two glancing regions $|\tau|=c_p|\xi'|_g$ and $|\tau|=c_s|\xi'|_g$. For the purpose of the inverse problem, it is enough to analyze the propagation of singularities away from a set of measure zero. Therefore, there is no need to build a parametrix near the glancing regions  (as in  \cite{MR1334206} or \cite{Yamamoto_09}, for example) or work as in \cite{Hansen84};  so we can avoid the glancing regions. 

By the calculus of the wave front sets, the traces of microlocal P waves on $\R\times\bo$ have wave front sets in the hyperbolic region under the assumption that all singularities hit the boundary transversely. The traces of transversal microlocal S waves belong to $c_s|\xi'|_g<  |\tau|$, i.e, either to the hyperbolic, the mixed one, or to the p-glancing one. In particular, the trace of any solution of the elastic system with singularities hitting transversely, has wave front disjoint from the elliptic region. On the other hand, boundary values of solutions of the boundary value or the transmission problem may have wave front set on that surface, as Rayleigh and Stoneley  waves do.

The analysis we have done so far, see next section, allows us to decouple the P and the S modes on the boundary completely by their polarizations. Then in terms of the potentials $w^s$ and $w^p$, we can think of the system as a decoupled one. When modes hit a free boundary, or a transparent one, however, the reflected and the transmitted modes may change type. The reason for this is that the boundary trace of an incoming S or P wave does not belong to the same subspace as that of an outgoing one. 


 \subsection{Wave front set in the hyperbolic region} \label{sec_ED1}
Let $f(t,x')$ be supported near some $(t_0,x_0')\in \R\times \R^{2}$, where $\R^{2}$ represents $\bo$, flattened. Assume first that $\WF(f)$ is supported in the hyperbolic region. The later has two disconnected components determined by the sign of $\tau$ there. Let us assume that $\WF(f)$ is contained with the one with $\tau<0$; the $\tau>0$ case is similar. Then the characteristic varieties reduce to $\tau+c_p|\xi|_g=0$ and $\tau+c_s|\xi|_g=0$, respectively. 
We are looking for a parametrix of the outgoing solution of the form $u = Uw = u_p+u_s$ 
as in \r{5.11} with $w$ a potential. Note that this construction excludes   $\xi_3=0$, which in our case corresponds to tangent rays which we avoid. We will work in a conic open  microlocal region which does not contain such rays, i.e., $\xi_3\not=0$ there.

We seek the potentials $w^p$ and $w^s$ as geometric optics solutions as in section~\ref{sec_Ac_BVP}, i.e., of the form \r{10c} (where the solution is called $u$, not $w$)   with  phases $\phi_p$  and $\psi_s$, respectively, and a scalar amplitude $a^p$   and a 2D vector-valued one $a_s= (a_1^s, a_2^s)$. The phase functions solve the eikonal equations
\be{11}
\partial_t\phi_p + c_p|\nabla_x\phi_p|_g=0, \quad \phi_p|_{x^3=0}=t\tau+ x'\cdot\xi',
\ee
and similarly for $\phi_s$, where $x'=(x^1,x^2)$. The choice of the positive sign in front of the square root in the eikonal equation is determined by the choice $\tau<0$. 
By \r{5.13}, the principal part of the amplitude of $u_p$ is $(D_x\phi_p) a^p$ and that of $u_s$ is $D_x\phi_s \times (a^s_1,a^s_2,0)$. Restricted to the boundary, we have $\nabla_x\phi_p= (\xi',\xi_3^p)$, $\nabla_x\phi_s= (\xi',\xi_3^s)$, where 
\be{14}
\xi_3^p: = \sqrt{c_p^{-2}\tau^2-|\xi'|_g^2}, \quad  \xi_3^s: = \sqrt{c_s^{-2}\tau^2- |\xi'|_g^2}, \quad \text{for $x^3=0$}.
\ee
We will use the notation
\be{14n}
\xi^p := (\xi',\xi_3^p) , \quad \xi^s :=(\xi',\xi^s_3) .
\ee
Those are the codirections of the rays emitted from the boundary, see Figure~\ref{pic_HR}. The angles $\theta^p$ and $\theta^s$ with the normal satisfy Snell's law
\be{Snell}
\frac{\sin\theta^p}{\sin\theta^s} = \frac{c_p}{c_s}>1, 
\ee
as it follows directly from \r{14}, see also \cite{SU-thermo_brain}.

\begin{figure}[!ht]
\includegraphics[page=4,scale=1]{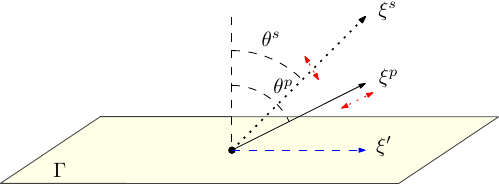}
\caption{The Dirichlet problem for the outgoing solution with wave front in the hyperbolic region.  There are emitted S and P waves.
}\label{pic_HR}
\end{figure}

As we stated above, we are going to do all principal  symbol calculations at $(t_0,x_0')$, where $g$ can always be arranged to be Euclidean.

In the hyperbolic region we work in, the expressions under the square roots are positive.  
The positive square roots guarantee that the singularities are outgoing.   We  determine next the boundary conditions for the transport equations. 
Since $u=Uw$, the boundary values of $w$ can be obtained from those of $u$ given by $f$ by an application of a certain \PDO. By the ``fundamental lemma'', see \cite[VIII.7]{Taylor-book0} and \cite{Treves2}, $Uw$ near the boundary is given by an oscillatory integral of the type \r{10c} with the amplitude there multiplied by a classical symbol with principal part $U(x,\nabla_x\phi)$, where $\phi$ equals either $\phi_s$ or $\psi_p$ depending on which components of $w$ we take. Restricted to the boundary, we get 
\be{fuU}
f=  u|_{x^3=0} =U_\textrm{out}\left(w |_{x^3=0}\right)
\ee
with $U_\textrm{out}$ a  classical    \PDO\  on $\R_t\times \R^2_{x'}$ with  principal symbol 
\be{Ub}
\sigma_p(U_\textrm{out}) = \begin{pmatrix}    0&-\xi_3^s&\xi_1\\  \xi_3^s&0&\xi_2\\ -\xi_2 & \xi_1&\xi_3^p\end{pmatrix}.
\ee
The subscript ``out'' is a reminder that we used the outgoing solution to define $U_\textrm{out}$. Similarly, we define $U_\textrm{in}$ using the incoming $u$. Its principal symbol is as above but with $\xi_3^s$ and $\xi_3^p$ having opposite signs.  Note that $U$ acts locally in $\R_t\times\R_x^3$ while the two new operators act on $\R_t\times \R_x^2$. 
The symbol $\sigma_p(U_\textrm{out})$ is elliptic, in fact   
\be{detU}
\det\sigma_p(U_\textrm{out})=  \xi_3^s(|\xi'|^2+\xi_3^s\xi_3^p),
\ee
which also equals $\xi^s_3 \langle\xi^s, \xi^p\rangle$. 
The inverse of $\det\sigma_p(U_\textrm{out})$ is easy to compute and we do that below. 
To find the boundary conditions for $w=(w^s_1,w^s_2,w^p)$, we write $w|_{x^3=0}=U_\textrm{out}^{-1}f$ (recall that all our inverses are parametrices). Then for $w^p$ and $w^s$ we get \r{5.12} with $\xi_3$ in all symbols replaced by $\xi_3^p$ for $u^p$ and $\xi_3^s$ for $u^s$.  Once we have the boundary conditions for $w$, we construct $w$ near the boundary by the geometric optics construction \r{10c}.  To get $u=u^p+u^s$, we apply $U$ to the result, see \r{5.11}. 

\begin{remark}\label{rem_Rachele}
In \cite{Rachele03}, Rachele showed that when $g$ is Euclidean,  the leading amplitudes (polarizations) of $u^p$ and $u^s$ are independent of $\rho$ if we think of the three parameters being $(\rho, c_s,c_p)$ instead of $(\rho,\mu,\lambda)$. We will use this in Section~\ref{section_GC}. 
\end{remark}

In what follows, we will make the calculations above more geometric.  
By \r{5.13}, $u^s$ and $u^p$ have representations of the kind \r{10c} with the corresponding phase functions and matrix valued amplitudes having principal parts $f\to \xi\times  (A^sf,0)$ and $f\to \xi  A^p\cdot f$, where $A_p$ is a  vector, and $A_s$ is a $2\times3$ matrix. 
Then one can show that on the boundary, $h\mapsto \xi^s\times (A_sh,0)$  is  the non-orthogonal projection to the plane $(\xi^s)^\perp$ parallel to  $\xi^p$, and  $h\mapsto \xi^p A_p\cdot h$ is the non-orthogonal projection to  $\xi^p$ parallel to the latter plane.  In other words, they are the projection operators related to the direct sum $\xi^p \oplus (\xi^s)^\perp$.

Finally in this section, we notice that the same analysis holds for the incoming solutions with given Dirichlet boundary data. Then in the formulas above, we have to take the negative square roots of $\xi_3^p$ and $\xi_3^p$ in \r{14}.

\subsection{Wave front set in the mixed region} \label{sec_ED1m}
Let $\WF(f)$ be in the mixed region next. We show below that the outgoing solution has a microlocal S wave only. The eikonal equation for $\phi_s$ still has the same real valued solution locally, corresponding to the outgoing choice of the solution $u_s$. On the other hand, the eikonal equation \r{11} for $\phi_p$ has no real solution. 
Indeed, we have $\nabla_{t,x'}\phi_p=(\tau,\xi')$ on $x^3=0$ and there is no real-valued function $\phi_p$ that could solve \r{11} and have such a gradient because in \r{14},   $\xi_3^p$ would be pure imaginary. 
This is the case of an evanescent mode described in Section~\ref{sec_Ac_evan}. 
 
We are still looking for a solution of the form $u=u_s+u_p = U(w^s,w^p)$ but this time $w_p$, and therefore, $u_p$  is an evanescent mode as the one constructed in Section~\ref{sec_Ac_evan}. The eikonal equation for $\phi_p$ implies, see  \r{6.9aa}, that $\xi_3^p$ in this case reduces to 
\be{xi_3pi}
\xi_3^p = \i\sqrt{ |\xi'|_g^2-c_p^{-2} \tau^2  }.
\ee
Then as in \r{fuU}, \r{Ub}, applying the ``fundamental lemma'' for FIOs with a complex phase, see \cite[X.4]{Treves2}, we deduce as before that the boundary values for $w$ are given by \r{fuU} with a classical \PDO\ $U$ having principal symbol as in \r{Ub} (with the new pure imaginary $\xi^p_3$).  The operator $U$ is still elliptic because the determinant \r{detU} has non-zero imaginary part. Then we can determine the boundary conditions for $w^s$ and $w^p$, construct the microlocal solutions, and apply $U$ to get $u$.

\subsection{Wave front set in the elliptic region} \label{sec_6.3}
Assume that $\WF(f)$ is in the elliptic region. Then we proceed as before, looking for both  $w^s$ and $w^p$ as evanescent modes with complex phase functions. In this case, both $\xi_3^p$ and $\xi_3^s$ are pure imaginary with positive imaginary parts, see \r{xi_3pi}, and for $\xi_3^s$ we get
\be{xi_3si}
\xi_3^s = \i\sqrt{ |\xi'|_g^2-c_s^{-2} \tau^2  }.
\ee
 We have
\[
\det\sigma_p(U_\textrm{out})=|\xi'|_g^2 -\sqrt{ |\xi'|_g^2-c_s^{-2} \tau^2  } \sqrt{ |\xi'|_g^2-c_s^{-2} \tau^2  }>0. 
\]
Therefore, $U_\textrm{out}$ is elliptic and we can proceed as above and construct the solution as in Section~\ref{sec_Ac_evan}.

\subsection{Summary} We established that the Dirichlet problem is well posed microlocally and we have the following:
\begin{itemize}
\item[(i)] $\WF(f)$ in the hyperbolic region: there are outgoing P and S waves. 
\item[(ii)] $\WF(f)$ in the mixed region: there is an  outgoing   S wave only (plus an evanescent P  mode).  
\item[(iii)] $\WF(f)$ in the elliptic region: there are no outgoing  waves; there are two evanescent modes.  
\end{itemize}

\section{The boundary value problem for the elastic system. Neumann boundary conditions and the Neumann operator} \label{sec_el_bvpN} 

Assume now that we want to find the outgoing solution of the elastic wave equation with boundary data $Nu=h$. The strategy below is find the Dirichlet boundary data $f$ from this equation and then to proceed as in section~\ref{sec_el_bvp}. In other words, we want to solve $\Lambda f=h$ for $f$ microlocally if possible by showing that $\Lambda$ is elliptic (or not). Lack of ellipticity of $\Lambda$ in the elliptic region leads to Rayleigh waves, see, e.g., \cite{popov_elast, Taylor-Rayleigh, MR1376435,MR1334206}. 

\subsection{Wave front set in the hyperbolic region} 

We are looking again for an outgoing  solution of the type $u=u^s+u^p$ as in \r{5.11}. The boundary values $w_b = w |_{x^3=0}$ of $w$ are computed by solving
\be{fuUN}
h= Nu|_{x^3=0},  =M_\textrm{out}w_b,\quad M_\textrm{out} := \Lambda U_\textrm{out}
\ee
for $w_b$, compare with \r{fuU}, where $\Lambda$ is the microlocalized Dirichlet-to-Neumann map \r{2a}, i.e., $\Lambda h: = Nu|_{x^3=0}$ 
for $u$ an outgoing microlocal solution of the elasticity equation with boundary data $u=h$ on $x^3=0$. We can use \r{fuUN} and \r{Ub} to compute $\sigma_p(\Lambda)$. 

We define the incoming $M_\textrm{in}$ in a similar way as in \r{fuUN} but with $u$ being the incoming solution. More precisely, $M_\textrm{in}w_b$ is defined as $Nu|_{x^3=0}$ where $u$ is the incoming solution with boundary data $U_\textrm{in}w_b$. This also means that $M_\textrm{in}= \Lambda_\text{in} U_\textrm{in}$, where $\Lambda_\text{in}$ is defined as $Nu|_{x^3=0}$ with $u$ being the incoming solution. The operator $\Lambda$ the should be denoted by $\Lambda_\text{out}$ but we will keep the simpler one $\Lambda$. Below, we compute the principal symbols of $M_\text{out}$ and $M_\text{in}$.  Combining that with \r{Ub}, we can compute the principal symbol of $\Lambda$ as well but we will not need it.

By \r{Nu} and \r{Ub}, in semigeodesic coordinates, 
\be{Nb0}\begin{split}
\sigma_p(M_\textrm{out}) &=  \begin{pmatrix}   \mu \xi_3^s& 0&\mu \xi_1\\ 0&\mu \xi_3^s& \mu \xi_2\\ \lambda \xi_1 & \lambda \xi_2&(\lambda+2\mu)\xi_3^s\end{pmatrix}
\begin{pmatrix}    0&-\xi_3^s&0\\  \xi_3^s&0&0\\ -\xi_2 & \xi_1&0\end{pmatrix}\\
& + \begin{pmatrix}   \mu \xi_3^p& 0&\mu \xi_1\\ 0&\mu \xi_3^p& \mu \xi_2\\ \lambda \xi_1 & \lambda \xi_2&(\lambda+2\mu)\xi_3^p\end{pmatrix}
\begin{pmatrix}    0&0&\xi_1\\  0&0&\xi_2\\ 0 & 0 &\xi_3^p\end{pmatrix}.
\end{split}
\ee
Therefore, 
\be{Nb}\begin{split}  
 \sigma_p(M_\textrm{out})&= \begin{pmatrix} -\mu \xi_1\xi_2& \mu(2 \xi_1^2 +\xi_2^2 ) -\rho {\tau}^{2}&2 \mu \xi_1\xi_3^p
\\  -\mu (\xi_1^2+2 \xi_2^{2})+ \rho \tau^{2}&\mu\xi_1\xi_2&2 \mu \xi_2 \xi_3^p 
\\ -2\mu\xi_2\xi_3^s&2\mu \xi_1\xi_3^s&-2 \mu |\xi|^{2}+\rho\tau^{2}\end{pmatrix} .
\end{split}
\ee
Similarly, we define  $ M_\text{in}$ to be the principal symbol of the same operator but related to the incoming DN map; i.e., the same as above but with $\xi_3^s$ and $\xi^p_3$ the negative square roots in \r{14}. 

A direct computation yields 
\be{RNC_det}
\begin{split}
\det  \sigma_p(M_\text{out}) &= - \left(    \mu |\xi'|^2-\rho{\tau}^{2} \right)  \left( 4|\xi'|^2  
\mu^2\left( \xi_3^p \xi_3^s+|\xi'|^2\right) -4\mu \rho{\tau}^{2
}|\xi'|^2 +\rho^2{\tau}^{4}
 \right) \\
 &= - \rho\left(    c_p^2 |\xi'|^2-{\tau}
^{2} \right)   \left(( 2\mu |\xi'|^2 -\rho \tau^2)^2  +4 \mu^2|\xi'|^2  
\xi_3^p \xi_3^s
 \right)>0.
 \end{split}
\ee
The determinant of  $\sigma_p(M_\text{in})$ is the same. 
Since $U_\textrm{out}$ is elliptic, we get that $\Lambda$ is elliptic in the hyperbolic region as well. Therefore, we can invert $\Lambda$ microlocally and reduce the Neumann boundary value problem to the Dirichlet one, which can be solved as in section~\ref{sec_ED1}. More directly, we invert $\Lambda U_\textrm{out}$ and we get boundary conditions for $w$; which we use to solve the problem.

\subsection{Wave front set in the elliptic region.} \label{sec_7.2} 
In this case, we seek both $w^s$ and $w^p$ as evanescent modes. The calculations are as in section~\ref{sec_el_bvp} but $\xi_3^s$ and $\xi_3^p$ are pure imaginary as in \r{xi_3pi} and \r{xi_3si}. Then
\be{7.3}
\det \sigma_p(M_\textrm{out}) = - \rho\left(    c_p^2 |\xi'|^2-{\tau}
^{2} \right)   \left(( 2\mu |\xi'|^2 -\rho \tau^2)^2  -4|\xi'|^2  
 {\mu}^{2}|\xi_3^p| |\xi_3^s|
 \right). 
\ee
We have $ c_p^2 |\xi'|^2-{\tau}^{2}>0$. For the third factor above, introduce the function 
\[
R(s) = \left(s-2\right)^2 - 4\left(1-s\right)^\frac12 \left(1-c_s^2c_p^{-2}s\right)^\frac12.
\]
Then, up to an elliptic factor, $\det \sigma_p(M_\textrm{out})$ equals   $R(c_s^{-2}\tau^2|\xi'|^{-2})$. It is well known and can be proven easily that on the interval $s\in (0,1)$, this function has a unique simple root $s_0$. This corresponds to $s_0c_s^2 |\xi'|^2=\tau^2$. Therefore, if we set $c_R(x)=c_s\sqrt{s_0}$, known as the Rayleigh speed, we get a  characteristic variety 
\be{S_R}
\Sigma_R := \left\{c_R^2|\xi'|^2_g=\tau^2\right\}
\ee
 on which \r{7.3} has a simple zero. Note that $0<c_R<c_s<c_p$. 
Since $U_{\rm out}$ is elliptic here, see Section~\ref{sec_6.3}, we get that $\Lambda$ is elliptic in the elliptic region away from $\Sigma_R$ and its principal symbol has a simple zero there. This generates the Rayleigh waves, see Section~\ref{sec_Rayleigh}. For every $f$ with $\WF(f)$ in the elliptic region but away from $\Sigma_R$, we can proceed as above to solve the Neumann problem.

\subsection{Wave front set in the mixed region} In this case, we seek both $w^s$ as a hyperbolic wave and $w^p$ as an evanescent one. The calculations are as in section~\ref{sec_el_bvp} with $\xi_3^s$ real as in \r{14} and $\xi_3^p$  pure imaginary as in \r{xi_3pi}. Then $ c_p^2 |\xi'|^2-{\tau}^{2}>0$ as well and for $\det \sigma_p(M_\textrm{out}) $ we have an expression similar to \r{7.3} given, up to an elliptic factor, by $R(c_s^{-2}\tau^2|\xi'|^{-2})$ with 
\be{RNC_detm}
R(s) = \left(s-2\right)^2 +4\i  (s-1)^\frac12 \left(1-c_s^2c_p^{-2}s\right)^\frac12.
\ee
For $1<s<c_p^2c_s^{-2}$, which corresponds to the mixed region, $R$ is elliptic. This shows that, as above, one can construct $w|_{x^3=0}$ microlocally given $\Lambda f$. Then we construct $w^s$ and $w^p$, the latter as an evanescent mode; and then $u$. In particular, only microlocal  S waves propagate from $\bo$. 

\subsection{Incoming solutions} The construction of incoming solutions (singularities propagating to the past only) is similar and we will skip the details. One can obtain them from the outgoing solutions by reversing the time.

\section{The boundary value problem for the elastic system. Cauchy data} \label{sec_el_bvpC} 
We analyze the boundary value problem for the elastic system on one side of $\Gamma$ with Cauchy data $u=f$, $\partial_\nu u=h$ on $\R_t\times \Gamma$. Similarly to section~\ref{sec_Ac_BVP_C}, we assume wave front set away from the glancing regions. This analysis is needed for the transmission problem when we want to control the behavior of the waves on one side by the other. We show in particular that this problem is well posed microlocally even though globally it is not, in general. 

\subsection{Wave front in the hyperbolic region} \label{sec_u_C}
Assume first that the wave front set of $(f,h)$ is in the hyperbolic region. We are looking for a solution   
\be{u_C1}
u = u_\text{in} + u_\text{out}=     (u_\text{in}^p  + u_\text{in}^s)+ (u_\text{out}^p  + u_\text{out}^s),
\ee
having both an incoming and an outgoing part, see Figure~\ref{pic_C}. 
\begin{figure}[!ht]
\includegraphics[page=5,scale=1]{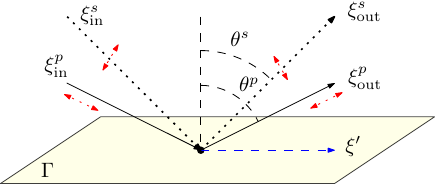}
\caption{The Cauchy problem with wave front in the hyperbolic region. The angle of incidence is the same as the angle of reflection for each type. Given any Cauchy data in the hyperbolic region, there is a unique solution (it is an elliptic problem). 
}\label{pic_C}
\end{figure}

Then on $\Gamma$, we need to solve
\be{u_C2}
  u_{\text{in},b} +  u_{\text{out},b} = f, \quad 
\Lambda_\text{in} u_{\text{in},b} + \Lambda_\text{out}u_{\text{out},b} = h,
\ee
for the boundary traces $u_{\text{in},b}$ and $u_{\text{out},b}$ of $u_{\text{in}}$ and $u_{\text{out}}$. We pass to the corresponding solutions $w$ as in \r{fuUN} to get
\be{u_C3}
U_\text{in}  w_{\text{in},b} +  U_\text{out}w_{\text{out},b} = f, \quad 
M_\text{in} w_{\text{in},b} + M_\text{out}w_{\text{out},b} = h.
\ee
Let $(a_{1,\text{in}}^s, a_{2,\text{in}}^s, a_\text{in}^p)^T  $ be the principal amplitude of $w_\text{in} $ and similarly for $w_\text{out} $. By the rotational invariance w.r.t.\ rotations in the $(\xi_1,\xi_2)$ plane (we justify this later), we can assume $\xi_2=0$. Then by \r{Nb}, 
\be{u_C3a}
\begin{split}
\sigma_p(M_\text{out})\big|_{\xi_2=0} &= \begin{pmatrix} 0 &2 \mu\xi_1^2-\rho {\tau}^{2}&2 \mu \xi_1\xi_3^p
\\  -\mu \xi_1^2+ \rho \tau^{2}&0&0
\\ 0&2\mu \xi_1\xi_3^s&-2 \mu \xi_1^{2}+\rho\tau^{2}\end{pmatrix}, \\ 
\sigma_p(U_\textrm{out})\big|_{\xi_2=0}  &= \begin{pmatrix}    0&-\xi_3^s&\xi_1\\  \xi_3^s&0&0\\ 0 & \xi_1&\xi_3^p\end{pmatrix},
\end{split}
\ee
and similarly for $\sigma_p(M_\text{in})$,  $\sigma_p(U_\text{in})$. Then  on principal symbol level, \r{u_C3} decouples into the following two systems
\be{u_C4}
A_\text{in} ( a_\text{in}^p, a_{2,\text{in}}^s)^T + A_\text{out} (a_\text{out}^p, a_{2,\text{out}}^s)^T = (\hat f_1,\hat f_3,\hat h_1, \hat h_3)^T,
\ee
and 
\be{u_C4a}
\begin{pmatrix}
\xi_{3}^s &- \xi_{3}^s \\
\mu(\xi_{3}^s)^2 & \mu(\xi_{3}^s)^2 \end{pmatrix} 
\begin{pmatrix}a_{1,\text{in}}^s\\ a_{1,\text{out}}^s \end{pmatrix}
= \begin{pmatrix}\hat f_2\\  \hat h_2\end{pmatrix} ,
\ee
where 
\be{u_C5}
A_\text{in}:= \begin{pmatrix}
\xi_1&-\xi_{3}^s \\
\xi_{3}^p&\xi_1 \\
2\mu\xi_{3}^p \xi_1 & \mu(2\xi_1^2-c_{s}^{-2}\tau^2)  \\
 - \mu( 2\xi_1^2-c_{s}^{-2}\tau^2 )  &   2\mu \xi_{3}^s\xi_1 
\end{pmatrix},
\quad
A_\text{out}:= \begin{pmatrix}
\xi_1&\xi_{3}^s \\
-\xi_{3}^p&\xi_1 \\
-2\mu\xi_{3}^p \xi_1 & \mu(2\xi_1^2-c_{s}^{-2}\tau^2)  \\
 - \mu( 2\xi_1^2-c_{s}^{-2}\tau^2 )  &  - 2\mu   \xi_{3}^s\xi_1 
\end{pmatrix}.
\ee
We have
\be{u_C6}
\begin{split}
\frac12(A_\text{in}+ A_\text{out}) &= \begin{pmatrix}
\xi_1&0 \\
0&\xi_1 \\
0 & \mu(2\xi_1^2-c_s^{-2}\tau^2)  \\
 - \mu(2\xi_1^2-c_s^{-2}\tau^2)   &  0
\end{pmatrix},\\
\frac12(A_\text{out}- A_\text{in})&= \begin{pmatrix}
0&\xi_{3}^s \\
-\xi_{3}^p& 0 \\
-2\mu\xi_{3}^p \xi_1    & 0 \\
 0 &  -2\mu\xi_{3}^s \xi_1 
\end{pmatrix}.
\end{split}
\ee
This shows that the system \r{u_C4}  decouples to two $2\times 2$ systems after rewriting it as a system for the sum and the difference of the original vectors.  The determinants of those two systems are $c_s^{-2}\tau^2\xi_3^p$ and $c_s^{-2}\tau^2\xi_3^s$, respectively; therefore, elliptic (after applying an elliptic operator of order $-1$ to the last two rows to equate their order with the rest, and we will use this notion of ellipticity below as well).  Therefore, \r{u_C4} is elliptic as well. Clearly, so is \r{u_C4a}, which behaves as the acoustic case \r{7.1}. Thus we proved  the following. 

\begin{lemma}\label{lemma_A}
The matrix valued symbol $(A_\text{\rm in}, A_\text{\rm out})$ is elliptic. 
\end{lemma}

Therefore, \r{u_C3} is elliptic as well.

Lemma~\ref{lemma_A} remains true in the mixed and in the elliptic regions as well, where $\xi_3^s$ or $\xi_s^p$ could be pure imaginary as in \r{xi_3pi}, \r{xi_3si}. Then there is no incoming/outgoing choice of the sign of $\xi_3^s$ and $\xi_s^p$ (which distinguishes $A_\text{in}$ and $A_\text{out}$) but this does not matter because later, we will multiply those expressions, when pure imaginary, with the ``wrong'' signs by zero, see \r{u_C4''}, for example. 
 
\subsection{SV-SH decomposition of S waves} \label{sec_SV-SH}
The principal amplitude of the S wave $u^s=D\times (w_1^s,0,0) = (0,D_3,-D_2)w_1^s$ (plus smoother terms),  
see Proposition~\ref{pr1} and \r{5.12}, corresponding to $w_2^s=0$, evaluated for $\xi_2=0$, has only its second component possibly non-zero. Then it is tangent to $\Gamma$ and normal to the direction of the propagation $\xi=(\xi_1,0,\xi_3)$ (as it should be because it is an S wave). 
In the geophysical literature  (for constant coefficients and a flat boundary),  such waves 
are called \textit{shear-horizontal} (SH) waves since their polarization is tangent to the plane $\Gamma$. Equation \r{u_C4a} then describes the SH waves generated by the Cauchy data when $\xi_2=0$. Note that in our case, ``horizontal'' makes sense only at the boundary. 

The $a_2^s$ terms appearing in \r{u_C4} are the \textit{shear-vertical} (SV) components of the potentials $w$ of the incoming and the outgoing waves. Indeed, using the subscript $b$ to indicate a boundary value (as we did above), when $w_{1,b}=0$, then the principal term of the outgoing/incoming  $u^s_b$ is $(\mp\xi_3(x',D'),0,D_1) w_{2,b}^s$, which gives us a principal amplitude perpendicular to the $\xi_2$ axis (and to the direction $\xi$ of propagation, of course).  Then the oscillations happen in the $\xi_1\xi_2$ plane, vertical to $\Gamma$ (and parallel to $\xi$), hence the name. System \r{u_C4} then describes how the SV and the P waves are created from given Cauchy data.

So far, the computations were done at a fixed point $x_0$ and a fixed covector $\xi^0$ at it,  where the metric is chosen to be Euclidean. Then the orthogonal projection of the principal amplitude to $\Gamma=\{x^3=0\}$ (actually, to $T_{x_0}^*\Gamma$) is the SH component of it, while the projection to the plane through it and the normal is the SV component. We will do this decomposition microlocally near $(x_0,\xi^0)$ on the principal symbol level. 

Note first that at $x_0$, there is a rotational invariance in the $\xi_1\xi_2$ plane. We already have a confirmation of that since we are free to choose coordinates in which $\xi_2=0$ and then we found out that the geometry of the rays and their principal amplitudes depend on the angles with the normal but not on $\xi$ in any other  way. To derive this, we conjugate  both symbols in \r{u_C3a} with the rotational matrix
\be{u_C7}
V:=  \begin{pmatrix}    \xi_1/|\xi|&\xi_2/|\xi|&0\\  -\xi_2/|\xi|&\xi_1/|\xi|&0\\ 0 & 0&1\end{pmatrix}.
\ee
A direct computation yields
\be{u_C8}
V^{-1} \sigma_p(M_\text{out}){(|\xi|,0}) V = \sigma_p(M_\text{out})(\xi), \quad V^{-1} \sigma_p(U_\text{out}){(|\xi|,0}) V = \sigma_p(U_\text{out})(\xi)
\ee
at $x=x_0$. So far, we assumed that the metric was Euclidean at $x_0$. To get that, one can set $\tilde \xi =g^{-1/2}(x_0)\xi$ which can be achieved by a linear change in the $x$ variables; then the Euclidean product in the $\tilde\xi$ variable corresponds to the metric one in the original $\xi$ one. Therefore, replacing $\xi$ above by $g^{1/2}(x_0)\xi$ gives us the principal symbols in the original local coordinates. Varying the point $x_0$, we get principal symbols locally. 

This allows us to define an SV-SH decomposition of S waves on a principal symbol level. In Proposition~\ref{pr1}, if $u^s$ is the S wave  of a solution with certain Cauchy data at $t=0$, then $u^s$ will be an SH wave on $\Gamma$ (up to lower order terms) if $\langle \nu, u^s\rangle|_{\Gamma}= 0$ up to lower order terms applied to the Cauchy data, where $\nu$ is a unit normal covector field. It would be an SV wave on $\Gamma$ if $\langle \nu, (D\times u^s) \rangle|_{\Gamma}= 0$, up to lower order. 
An outgoing S wave $u_\text{out}^s$ near $\Gamma$, which is determined uniquely (up to a smooth term) by its Dirichlet data on $\Gamma$; and therefore by its potential $w_{\text{out},b}$ on $\Gamma$, is an SV wave on $\Gamma$, if $D'\times w_{\text{out},b}=0$ up to a first order \PDO\ applied to $w_{\text{out},b}$, which corresponds to the requirement that the second  component of $w_{\text{out},b}$ must vanish when $\xi'=(\xi_1,0)$. Here, $D'$ is the tangential differential. 
To construct such SV waves, one can take the gradients on $\Gamma$ of  scalar functions with  non-trivial wave front sets.  
The $u^s$ wave is an SH one on $\Gamma$, if $D'\cdot  w_{\text{out},b}=0$ up to a lower order (divergence free). To construct such SH waves, one can take the curl on $\Gamma$ of  scalar functions with  non-trivial wave front sets.  

\subsection{Wave front in the mixed  region} The P wave is evanescent, and there is only one (not incoming and an outgoing one). The number of the unknown amplitudes on the boundary is reduced by one, and the system can be seen to be over-determined. Indeed, then $\xi_3^p$ is pure imaginary and given by \r{xi_3pi}. We still define $A_\textrm{in}$ and $A_\textrm{out}$ as in \r{u_C5}. 
Then \r{u_C4} becomes
\be{u_C4'}
A_\text{in} (0,a_{2,\text{in}}^s)^T + A_\text{out} ( a^p, a_{2,\text{out}}^s)^T = (\hat f_1,\hat f_3,\hat h_1, \hat h_3)^T,
\ee
and \r{u_C4a} stays the same. 
By the expressions of the determinants following \r{u_C6}, the matrix $(A_\text{in}, A_\text{out})$ is still elliptic in this case, i.e., Lemma~\ref{lemma_A} still holds. System \r{u_C4'} then is over-determined and solvable (uniquely) only if the r.h.s.\ belongs to a certain 3D subspace.

\subsection{Wave front  in the elliptic region} In this   case, $\xi_3^s$ is pure imaginary as well as in \r{xi_3si}, both waves are evanescent and the problem is overdetermined, as well. Equation \r{u_C4'} reduces to
\be{u_C4''}
 A_\text{out} ( a^p, a_{2}^s)^T = (\hat f_1,\hat f_3,\hat h_1, \hat h_3)^T,
\ee
and  Lemma~\ref{lemma_A} still holds with both $\xi_3^p$ and $\xi_3^s$ pure imaginary as in \r{xi_3pi}, \r{xi_3si};  therefore we get an overdetermined system as well. In system \r{u_C4a}, both amplitudes are equal and that system is overdetermined as well.

\section{Reflection and mode conversion  of S  and P waves from a free boundary with  Neumann boundary conditions} \label{sec_ref}
Let $\Gamma$ be a surface which separates an elastic medium from a free space (like the Earth from air). The natural boundary condition then is 
\be{RMC1}
Nu=0\quad\text{on $\Gamma$},
\ee
which means zero traction on $\Gamma$, i.e., no normal force, because the exterior has zero stiffness. We study reflection and mode conversion of S  and P waves when they come from the elastic side of $\Gamma$ and hit $\Gamma$. 

This is actually a partial case of the analysis of the boundary value problem with Cauchy data in Section~\ref{sec_el_bvpC} with zero Neumann and Dirichlet data. 
The strategy is the following. We take the trace $Nu_I$ of the incoming wave $u_I$  on the boundary and look for a reflected wave as a sum of an S  and P wave as in \r{RMC2} below. Then $Nu_I$ determines Neumann boundary conditions for those two waves. If $Nu_I$ has a wave front set in the hyperbolic region, we can recover the Dirichlet data for the reflected wave by inverting the elliptic \PDO\ $\Lambda U_\textrm{out}$ in \r{Nb}. Knowing the Dirichlet data, we reduce the problem of constructing an outgoing solution as in section~\ref{sec_ED1}. If $\WF(Nu_I)$ is in the mixed region, we use the construction in section~\ref{sec_ED1m}. Finally, $\WF(Nu_I)$ cannot be in the elliptic region since it corresponds to an incoming solution; therefore, Rayleigh waves cannot be generated by reflection of S  and P waves. One can verify that the principal amplitudes of the reflected S  and P waves can only vanish for a discrete number of incident angles (i.e., on a finite number of curves on the sphere of directions) because they depend analytically on $\xi$ and one can easily eliminate the scenario of one of the waves to vanish for all incoming directions. Those principal amplitudes can actually be computed and in the case of constant coefficients and a flat boundary, they have been computed in the geophysics literature, see, e.g., \cite{aki2002quantitative}. They do have zeros. For our purposes, it is enough to express their solution by Cramer's Rule since we will prove that the determinant does not vanish. Vanishing amplitudes at finite number of angles  is not an obstacle  for the inverse problem we solve because the missing rays can be added to the data by continuity (but that may affect stability). 

\subsection{$\WF(u_{I,b})$ in the hyperbolic region} 
Assume that we have an incident P wave $u_I= u_I^p+u^s_I$, in other words a sum of microlocal solutions near $\Gamma$ with $\WF(u^p_I)\subset  \Sigma_p$ and $\WF(u^s_I)\subset  \Sigma_s$.   As in Section~\ref{sec_GO}, we will restrict the wave front set to $\tau<0$. We extend $u_I$  to a two sided neighborhood of $\Gamma$ as a microlocal solution by extending the coefficients $\lambda$, $\mu$ and $\lambda$ in a smooth way in the exterior. Set $u_{I,b}=u_I|_{\R\times S}$. It follows form the analysis above that $\WF(u_{I,b})$ is in the mixed region. As above, we assume no wave front set in the glancing region. In fact, $\WF(u_{I,b}^p)$ is in the hyperbolic region while $\WF(u_{I,b}^s)$ is there only if the angle of the corresponding rays with the normal is smaller than the critical one given by $c_p|\xi'|=|\tau|$, and it is in the mixed one if the incident angle is greater than the critical one. 


We look for a solution of the form
\be{RMC2}
u = u_I+u_R = (u_I^p+u_I^s) +  (u_R^p + u_R^s),
\ee
where $u_R^p$ and $u_R^s$ are reflected P  and S waves, respectively. 

Let $x=(x',x^3)$ be semigeodesic coordinates near $x_0=0$ so that $x^3>0$ on the elastic side. All equalities below are at a fixed point $x_0$ which can be chosen to be $0$ and modulo lower order terms for the amplitudes. As above, we assume without loss of generality that the metric $g$ is Euclidean at $x=0$ to simplify the notation. We can get the equations below by using \r{Nb}. Let $w_I=(w_{1,I}^s,w_{2,I}^s,w_I^p)$ and $w_R=(w_{1,R}^s,w_{2,R}^s,w_R^p)$ be the solutions $w$ as in  \r{5.11} related to $u_I$ and $u_R$. Since they solve \r{SP-dec}, each singularity of the S  or the P part of $w_I$ reflects by the laws of geometric optics. On the other hand, if $\theta^p$ is the angle which an incoming P singularity makes with the normal, then the corresponding angle $\theta^s$ of the reflected S singularity, see Figure~\ref{pic1}, is related to $\theta^p$ by  Snell's law \r{Snell} 
as it follows directly from \r{14}, see also \cite{SU-thermo_brain}. Also, the incoming and the outgoing directions, and the normal belongs to the same plane, which determines the reflected direction uniquely.  The same law applies to an incoming S wave generating a reflected P one. In the latter case, there is a critical incoming angle $\theta_\text{cr}=\arcsin(c_s/c_p)$ of an S wave so that if $\theta^s>\theta_\text{cr}$, \r{Snell} has no solution for $\theta^p$. Then a reflected P wave does not exist and instead we have an evanescent mode, as we show below. 

\begin{figure}[!ht]
\includegraphics[page=1,scale=1]{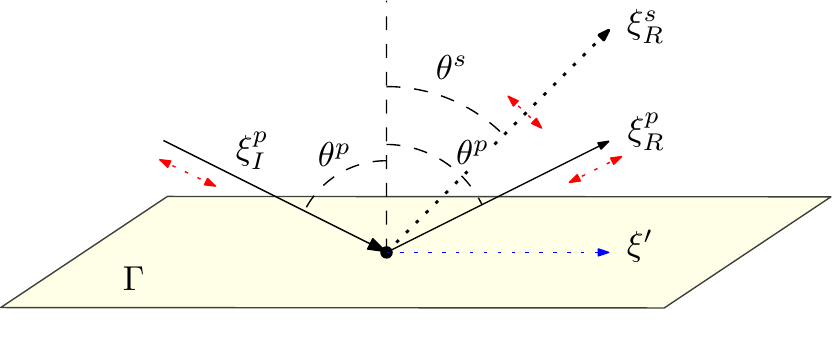}
\caption{Reflected P  and S waves from an incident P wave. The covectors shown  are parallel to the velocity vectors $c_p^2\xi_I^p$  of the incident P wave and the velocities  $c_p^2\xi_R^p$ and $c_s^2\xi_R^s$ of the reflected P  and S waves, respectively. The amplitudes depend on the type of the boundary condition.}\label{pic1}
\end{figure}

We need to solve 
\be{RMC2LU}
M_\text{out} w_{R,b}= -M_\text{in}w_{I,b}
\ee
for $w_{R,b}$. 
Since $M_\text{out}$ is elliptic in the hyperbolic region, \r{RMC2LU} is microlocally solvable. We only need to verify that $w_R$ has non-trivial S  and P components for almost all incoming rays. 

We express $w_I$, $w_R$ and $w_T$ in the form \r{10c} with phase functions solving \r{eik0} with for either $c_p$ or $c_s$ and a choice of the square root sign corresponding to the incoming or the outgoing property of each wave. The corresponding principal amplitudes are $(a_1^s, a_2^s,a^p)$  subindices   $I$, $R$, and $T$ distinguishing between the three waves.

Without loss of generality, we may assume $\xi_2=0$ as in section~\ref{sec_el_bvpC}. We get, see \r{u_C3a}, 
\be{RMC_11}
\begin{split} \left(2 \mu\xi_1^2 -\rho {\tau}^{2}\right)\left( a_{2,R}^s+ a_{2,I}^s \right)
+ 2 \mu \xi_1\xi_3^p \left( a_{R}^p- a_{I}^p \right)&=0,\\  
 2\mu \xi_1\xi_3^s\left( a_{2,R}^s- a_{2,I}^s \right)  -( 2 \mu \xi_1^{2}-\rho\tau^{2}) \left(a_{R}^p+ a_{I}^p \right)&=0,\\
 a_{1,R}^s+ a_{1,I}^s& =0.
\end{split} 
\ee
The system \r{RMC_11} is uniquely solvable, as we know. We determine $ a_{1,R}^s = - a_{1,I}^s$ first, which says that the  SH    wave $U(a_{1,R}^s,0)$ just flip a sign at reflection. The first two equations can be solved to get $a_{2,R}^s$ and $a^p_R$. If $a_R^p =a_{1,R}^s =0$, then $U(0,a_{2,R}^s)$ is the   SV   wave oscillating in the plane normal to the boundary. 

Let $w_I$ be a purely P wave, i.e.,   $w_{1,I}^s =w_{2,I}^s =0$. We want to find out when there is no reflected either P  or an S wave. One could just solve the system but we will analyze it without solving it. If there is no reflected P wave, i.e., if $w^p_R=0$, then \r{RMC_11} implies that both components of the reflected wave must vanish as well which is a contradiction, unless $2\mu\xi_1^2-\rho\tau^2=0$, i.e., if $2c_s^2\xi_1^2=\tau^2$. This may or may not be in the hyperbolic region and defines a cone of incoming directions when it does. Now, assume that there is no reflected S wave, i.e.,  $w_{1,R}^s =w_{2,R}^s =0$. This is possible only when $\xi_1=0$, i.e., when the incoming P wave is normal to the boundary. 

Now, assume that $w_I$ is an S wave. If there is a reflected S wave only, we are in the situation above with the time reversed --- it can only happen for normal rays.  Similarly, if there is a reflected P wave only, this can only happen for incident directions on a specific cone, or it does not. 

\subsection{Wave front set in the elliptic region, Rayleigh waves}\label{sec_Rayleigh}

We are looking for microlocal solutions satisfying  $Nu=0$ with wave front set on the boundary in the elliptic region. We follow Taylor \cite{Taylor-Rayleigh}, where the coefficients are constant and $n=2$ but as noted there, the construction extends to the general case; and will sketch that extension. 
As shown in Section~\ref{sec_7.2}, $\Lambda$ has a characteristic variety $\Sigma_R$, see \r{S_R} and the determinant of its principal symbol, up to an elliptic factor near $\Sigma_R$, is given by $H:=\tau^2-c_R^2|\xi'|^2$. Therefore, microlocal solutions to $Nu=0$ with boundary wave front sets on $\Sigma_R$ would solve a \PDO\ system on $\R_t\times\Gamma$ of real principal type in the sense of \cite{Dencker_polar}. Here, $|\xi'|$ is the norm of the covector $\xi'$ in the metric on $\Gamma$ induced by $g$ (the latter is Euclidean in the isotropic elastic case). 
 One can impose Cauchy data at $t=0$ to get unique (in microlocal sense) solution. Singularities propagate along the null bicharacteristics of $H$, i.e., along the null bicharacteristics of a wave equation on $\R_t\times\Gamma$ with speed $c_R$. 

Next, one uses the solution on $\R_t\times\Gamma$ constructed above as Dirichlet data for a solution near $\Gamma$, in $\Omega$, as in Section~\ref{sec_Ac_evan}.

 \subsection{Wave front set in the mixed  region} This can only happen if there is a non-zero incident S wave hitting the boundary at an angle (with the normal) greater than the critical one  $\theta_\textrm{cr}$, see  \r{theta_cr}. We are still looking for a solution of the kind \r{RMC2}, where $u_I^p=0$ and all singularities of $u_I^s$ hit the boundary at angles greater than  $\theta_\textrm{cr}$. Then $u_R^p$ would be actually an evanescent mode (not actually a P wave by our definition because it would be smooth away from $\Gamma$). To find the boundary values for $w_R$, we need to solve \r{RMC2LU} again with  $M_\text{out}$ as in \r{Nb} but $\xi_3^p$ is given by \r{xi_3pi}. The matrix $M_\text{out}$ is elliptic, see \r{RNC_detm}. Once we have the boundary values for $w_R$, we can construct both solutions as in section~\ref{sec_el_bvp}. 
We also see that the reflected S wave cannot have zero amplitude except for possibly one incident angle; the proof is like in the hyperbolic case. 
 
\subsection{Summary} 
\begin{itemize}
\item[(i)] An incident P wave produces a reflected P wave and a reflected S wave. 
\item[(ii)]  An incident S wave produces a reflected S wave. It produces a reflected P wave only if the incident angle is greater than the critical one; otherwise there is an evanescent P solution. 
\item[(iii)]  By time reversal, given an outgoing P wave, there are incoming S  and P  ones which produce that P wave and no S wave. The roles of those waves can be reversed only when the incident angle of the S wave is greater than the critical one. 
\item[(iv)]  An incident SH wave produces a reflected SH wave only.  
\end{itemize}

\section{The transmission problem for the elastic system} \label{sec_trans}
\subsection{Transmission and reflections of incoming S and P waves.  Zoeppritz' and Knott's equations} \label{sec_Tr_1}

We are interested first how an incoming wave, either and S or an P one, is reflected and transmitted across $\Gamma$. We assume first that the wave front set of the incoming waves on the boundary is in the hyperbolic region on the other side of $\Gamma$ as well. This is a classical case with a long history. 
As in section~\ref{sec_Ac_RT}, we assume that $\Gamma$ divides $\R^3$ locally into $\Omega_+$, where the waves come from, and $\Omega_-$, where they may transmit. 
 Let, as above, $u_I$ be a microlocal solution of the elastic system. 
Similarly to \r{14u}, we are looking for a local solution of the form
\be{u's}
u = u_I+u_R+u_S=   (u_I^p+u_I^s) + (u_R^p + u_R^s)+ (u_T^p+ u_T^s),
\ee
where the expressions in each parentheses is a decomposition into P and S waves, 
 $u_T$, is supported in $\bar\Omega_+$, and $u_I, u_R$ are supported in $\bar\Omega_-$. The terms with a superscript $s$ are microlocally S waves; and those with a superscript $p$ are P waves.  

Denote the restriction of $c_p$ and $c_s$ to $\bar\Omega_+$ and $\bar\Omega_-$, respectively by $c_{p,+}$ or $c_{s,+} $; and  $c_{p,-}$, $c_{s,-}$, respectively. A subscript $b$ denotes a boundary value. We know that  $\WF(u_{I,b})$ is in the hyperbolic or the mixed region on $T^*S$ w.r.t.\ the speeds $c_{p,+}$ and $c_{s,+}$ assuming non-trivial incoming S  and P waves. This may not be true  
on the negative side, i.e., with respect to the speeds  $c_{p,-}$ and $c_{s,-}$ but as we said above, in this section, we are assuming that $\WF(u_{I,b})$ is in the hyperbolic region with respect to them as well. 

\begin{figure}[!ht]
\includegraphics[page=14,scale=1]{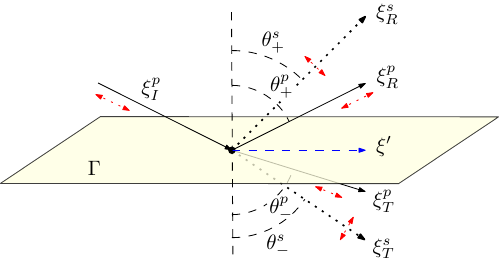}
\caption{The elastic transmission problem: Reflected and transmitted P and S waves from an incident P wave (the incoming S wave not shown). 
In this diagram, each speed gets faster in the lower half space which decreases the angles of the transmitted rays with $\xi'$ compared to the reflected ones or it would create evanescent modes. 
}\label{pic2}
\end{figure}

The transmission conditions  $[u]_\Gamma=0$, $[Nu]_\Gamma=0$ in \r{tr} are equivalent to
\be{RMC4.1}
\begin{split}
U^+_\text{in}w_{I,b}+ U_\text{out}^+w_{R,b} &= U^-_\text{out}w_{T,b},\\
M^+_\text{in}w_{I,b}+ M^+_\text{out}w_{R,b} &= M^-_\text{out}w_{T,b},
\end{split}
\ee
where the $\pm$ superscripts indicate that the corresponding operators act in $\Omega_\pm$. We will show next that this system is elliptic for recovery of $w_{R,b}$ and $w_{T,b}$ given $w_{I,b}$. In fact, ellipticity is a consequence of the energy preservation. Take the dot product of the two equations above (recall that we work at a fixed point where the metric is transformed to an Euclidean one). We get
\be{RMC_en0}
\langle U^+_\text{in}w_{I,b},M^+_\text{in}w_{I,b}\rangle + \langle U^+_\text{out}w_{R,b},M^+_\text{out}w_{R,b}\rangle =\langle U^-_\text{out}
w_{T,b},M^-_\text{out}w_{T,b}\rangle 
\ee
 because it can be shown that $(U^+_\text{in})^* M^+_\text{out}+ (M^+_\text{in})^* U^+_\text{out}=0$ up to smoothing terms. The latter can be proven in the following way. The quadratic form $\langle U_\textrm{in}^+ w_{I,b} , M_\textrm{in}^+ w_{I,b}\rangle$ is proportional to the energy flux of $u_I$ through $\R\times \Gamma$ as can be shown by integration by parts: we get $2\Re\int_{\R\times \Gamma} \langle u_t,\Lambda u\rangle$, see, e.g., \cite{SU-thermo_brain}. Similarly, the other two forms are proportional to energy fluxes, and the signs, after a multiplication by the same constant, are $+,-,+$.  
Then if $w_{I,b}=0$ (i.e., if \r{RMC4.1} is homogeneous), the signs of the forms imply the zero solution only. 
The cancellation  equality above reflects the fact that the incoming and the outgoing wave are microlocally separated. We are not going to prove it this way because below we will get a direct confirmation for the principal symbols, which is what we need. 

 In matrix form, that system is given by
\be{RMC4.1a}
\begin{pmatrix}  U^+_\text{out} & -U^-_\text{out}\\ M^+_\text{out} & -M^-_\text{out}\end{pmatrix}
\begin{pmatrix}  w_{R,b}\\ w_{T,b}\end{pmatrix}
= -\begin{pmatrix} U^+_\text{in}  w_{I,b}\\ M^+_\text{in}  w_{I,b}\end{pmatrix}.
\ee
We compute the principal symbol of the matrix operator applied to $(w_{R,b},w_{T,b})$. As in the previous section, we work at a fixed point where the boundary metric is chosen to be Euclidean. By the invariance under rotations in the $x^1x^2$ plane, we can perform the computations when $\xi_2=0$, as in the previous section. For the principal amplitude of $w$ on $\Gamma$, we will adopt the following notation: $(\SH,\SV,P)^T$, i.e., $\PP=a^p$ in the notation of the previous section, and $a^s=(\SH,\SV)$ is the decomposition of the principal amplitude of the potential (on the boundary) of  the S wave $u^s=D\times w^s$ into shear-horizontal and shear-vertical  terms. We use the subscripts $I,R,T$ for the same purpose as above. 

The system \r{RMC4.1a} then decouples into a $4\times 4$ one and a $2\times2$ one. The $4\times 4$ system has the form 
\be{RMC_1}
A_\textrm{in}^+(\PP_I , \SV_I)^T + A_\textrm{out}^+(\PP_R, \SV_R)^T= 
 A_\textrm{out}^-(\PP_T, \SV_T)^T.
\ee
We use the notations $A_\text{in}$ and $A+\text{out}$, see \r{u_C5} with plus or minus superscripts depending on which side of $\Gamma$ they are related to. By Lemma~\ref{lemma_A}, $(A_\text{in}, A_\text{out})$ is elliptic.

The second system, describing the reflection and the transmission of SH waves, is
\be{RMC4.2}
\begin{pmatrix}
\xi_{3,+}^s & \xi_{3,-}^s \\
\mu_+(\xi_{3,+}^s)^2 & -\mu_-(\xi_{3,-}^s)^2 \end{pmatrix} 
\begin{pmatrix}SH_R\\ SH_T \end{pmatrix}
= SH_I\begin{pmatrix}\xi_{3,+}^s \\-\mu_+(\xi_{3,+}^s)^2 \end{pmatrix} .
\ee
It has a negative determinant, therefore it is elliptic. This decoupling shows that the SH waves do not convert to other modes and reflect and transmit similarly to acoustic waves. We can write \r{RMC4.2} as
\[
\xi_{3,+}^s(\SH_R - \SH_I)=-\xi_{3,-}^s \SH_T, \quad  
\mu_+(\xi_{3,+}^s)^2(\SH_R + \SH_I)=\mu_-(\xi_{3,-}^s)^2 \SH_T.
\]
Multiply those  equations to get
\be{RMC101}
\rho_+c_{s,+}^2(\xi_{3,+}^s)^3\left(|\SH_R|^2 - |\SH_I|^2\right) +\rho_- c_{s,-}^{2} (\xi_{3,-}^s)^3| \SH_T|^2=0.
\ee
when all $w$'s are real. If they are complex, we can justify this by the equality $\Re (z-w)(\bar z+\bar w)=|z|^2-|w|^2$. Without going into details, we mention that this is actually an energy equality of the kind \r{RMC_en0} with $c_{s,\pm}^2$ normalization factors since the column vectors of $U$ in \r{U} are not normalized according to the corresponding speed, $\rho_\pm$ are volume element factors, the  $(\xi_{3,\pm}^s)^2$ factors come from the contribution of an S wave with principal term proportional to $(\xi_1,0,-\xi_3^s  )\times ( 1,0,0)= \xi_3^s(0,1,0)$ to $Nu$; and the extra $\xi_{3,\pm}^s$ factor accounts for the angle of incidence of reflection/transmission. 
Equations \r{RMC4.2} imply that when $\SH_I\not=0$, we have $\SH_T\not=0$; and  $\SH_R=0$ when $\mu_+ \xi_{3,+}^s  - \mu_- \xi_{3,-}^s$ which can happen for a fixed $|\xi'|$.  

We are going back to the system    \r{RMC_1}. 
We will transform it  into a form used in the geophysics literature. Let $\theta_+^p$, $\theta_+^s$, $\theta_-^p$ and $\theta_-^s$ be the angles between the normal and $\xi^p_R$, $\xi^s_R$, $\xi_T^p$ and $\xi_T^s$, respectively, see Figure~\ref{pic2}. Note that those angles are in $[0,\pi/2)$ and we exclude the zero ones below just to be able to put the equations into the desired form and to compare them with classical results. The singularity at $0$ can be resolved by multiplying the corresponding equations by the appropriate sine functions. Then 
\be{RMC_cot}
\xi^p_{3,+}/\xi_1=\cot\theta_+^s, \quad ( 2\xi_1^2-c_{s,+}^{-2}\tau^2 ) /\xi_1^2=1-(\xi_3^p)^2/\xi_1^2=1-\cot^2\theta_+^s, 
\ee
and similarly for the other angles. 

Divide the first two equations in \r{RMC_1} by $\xi_1$ and the last two by $\xi_1^2$, for $\xi_1\not=0$, to put the system in the form $A'\mathbf{a} =B'  \mathbf{b}$ with 
\[
A':= \begin{pmatrix}
1&-\cot\theta_{+}^s&-1&-\cot\theta_{-}^s\\
\cot\theta_{+}^p&1& \cot\theta_{-}^p&-1\\
2\mu_+\cot\theta_{+}^p & \mu_+ (1-\cot^2\theta_+^s) 
&2\mu_-\cot\theta_{-}^p & - \mu_- (1-\cot^2\theta_-^s) \\
 - \mu_+(1-\cot^2\theta_+^s)   &  2 \mu_+ \cot\theta_{+}^s &
   \mu_-(1-\cot^2\theta_-^s)  &    2\mu_-   \cot\theta_{-}^s
\end{pmatrix}
\]
and similarly, $B'$ is the most right $4\times 2$ block of $A'$ with all minus subscripts replaced by plus ones. Here, $\mathbf{a} =(\PP_R, \SV_R, \PP_T, \SV_T)^T $, $\mathbf{b}= (\PP_I, \SV_I)^T$.
  The resulting system is the   \textit{Knott's equations} \cite{Knott1899} derived by Knott in 1899 for a flat boundary and constant coefficients. The form here corresponds to \cite{Sheriff_Seismology}. We write them as  
\[
\begin{split}
(\PP_R+\PP_I)-\cot\theta_{+}^s( \SV_R - \SV_I  )&=w_T^p +\cot\theta_{-}^s  \SV_T, \\
\cot\theta_{+}^p (\PP_R-\PP_I)  +( \SV_R + \SV_I  )&= -\cot\theta_{-}^pw_{2,T}^p+\SV_T,\\
2\mu_+\cot\theta_{+}^p(\PP_R-\PP_I) + \mu_+ (1-\cot^2\theta_+^s) ( \SV_R + \SV_I  )&=  
-2\mu_-\cot\theta_{-}^p w_T^p+ \mu_- (1-\cot^2\theta_-^s)\SV_T , \\
 - \mu_+(1-\cot^2\theta_+^s)  (\PP_R+\PP_I)+   2\mu_+  \cot\theta_{+}^s ( \SV_R - \SV_I  ), &= 
 -  \mu_-(1-\cot^2\theta_-^s)w_T^p  -  2\mu_- \cot\theta_{-}^s \SV_T .
\end{split}
\]
Following \cite{Knott1899}, we multiply the corresponding sides of the first and  the third equations; then do the same thing with the second and the fourth one and add the results to get
\be{RMC_en}
\begin{split}
\mu_+\frac{\cot\theta_+^p}{\sin^2\theta_+^s}\left(|\PP_R|^2- |\PP_I|^2\right) &+
\mu_+\frac{\cot\theta_+^s}{\sin^2\theta_+^s}\left(|\SV_R|^2 - |\SV_I|^2\right)\\
&+ 
\mu_-\frac{\cot\theta_-^p}{\sin^2\theta_-^s}|\PP_T|^2+ 
\mu_-\frac{\cot\theta_-^s}{\sin^2\theta_-^s}|\SV_T|^2= 0,
\end{split}
\ee
therefore, 
\be{RMC7}
\begin{split}
\rho_+\cot\theta_+^p \left(|\PP_R|^2- |\PP_I|^2\right) &+ 
\rho_+ \cot\theta_+^s (|\SV_R|^2-|\SV_I|^2 )\\
& + 
\rho_- \cot\theta_-^p |\PP_T|^2+ 
\rho_- \cot\theta_-^s |\SV_T|^2= 0.
\end{split}
\ee
We used here that $\rho_+\sin^2\theta_+^s=(\xi_1^2/\tau^2)\mu_+$ and similarly for the other terms. 

As noted by Knott \cite{Knott1899}, this is an energy equality, stating that the sums of the energy fluxes of the four generated waves, on a principal symbol level, equals that of the incident one. It is also a version of \r{RMC101}.

Equation \r{RMC7} implies that the homogeneous system $A'\mathbf{a}=0$ has the zero solution only. Therefore, $A'$ is elliptic.  Explicit formulas for the solution of this system can be found in  \cite{aki2002quantitative} for the flat constant coefficient case, and those formulas generalize to our case once we make them invariant. 

\subsection{The general case with incoming waves from both sides} \label{sec_9.2} 
We assume waves coming from both sides, see Figure~\ref{pic2a} some of them possibly evanescent, with Dirichlet (and therefore Cauchy) data of their traces on $\Gamma$ in a small neighborhood of some covector in $T^*\Gamma$. 

 We classify the cases by hyperbolic-hyperbolic (HH), hyperbolic-mixed (HM), mixed-mixed (MM), mixed-elliptic (ME) and elliptic-elliptic (EE) according to the location of the wave front of the Cauchy data on the positive/negative side of $\Gamma$. 

 \subsubsection{The hyperbolic-hyperbolic (HH) case} \label{sec_HH}
Assume a wave front set in the hyperbolic region on both sides. 
 This is automatically true if on each side, we have both S  and P waves. The construction in section~\ref{sec_Tr_1} then generalizes directly. We are going to denote the incoming and the outgoing solutions $w$ on each side by $w_\text{in}^+$, $w_\text{out}^+$, $w_\text{in}^-$, $w_\text{in}^-$. The transmission conditions \r{tr} then take the form
\be{RMC4.1n}
\begin{split}
U^+_\text{in}w_\text{in,b}^++ U_\text{out}^+w_\text{out,b}^+ &=U^-_\text{in}w_\text{in,b}^-+ U_\text{out}^- w_\text{out,b}^-,\\
M^+_\text{in}w_\text{in,b}^++ M_\text{out}^+w_\text{out,b}^+ &=M^-_\text{in}w_\text{in,b}^-+ M_\text{out}^- w_\text{out,b}^-,
\end{split}
\ee
compare with \r{RMC4.1} and \r{u_C3}. 
We use the notation in section~\ref{sec_u_C} but we put superscripts $+$ and $-$ depending on the side of $\Gamma$ we work on. 
 We use the notation $ (\PP,\SV,\SH)$ as above for the principal amplitude of $w$ on $\Gamma$, with the corresponding subscripts and the superscripts. 
Then \r{RMC4.1n} decouples into the following two equations
\be{main_system}
A_\textrm{in}^+(\PP_\text{in}^+ , \SV_\text{in}^+)^T + A_\textrm{out}^+(\PP_\text{out}^+, \SV_\text{out}^+)^T= 
A_\textrm{in}^-(\PP_\text{in}^- , \SV_\text{in}^-)^T + A_\textrm{out}^-(\PP_\text{out}^-, \SV_\text{out}^-)^T
\ee
and 
\be{2x2}
\begin{pmatrix}  \Big. 
-\xi_{3,+}^s & \xi_{3,+}^s \\ 
\mu_+(\xi_{3,+}^s)^2 & \mu_-(\xi_{3,+}^s)^2 \end{pmatrix} 
\begin{pmatrix} \Big. \SH_\text{in}^+\\ \SH_\text{out}^+ \end{pmatrix}
=\begin{pmatrix}\Big. 
\xi_{3,-}^s & -\xi_{3,-}^s \\
\mu_+(\xi_{3,-}^s)^2 & \mu_-(\xi_{3,-}^s)^2 \end{pmatrix}
\begin{pmatrix}\Big. \SH_\text{in}^-\\ \SH_\text{out}^- \end{pmatrix},
\ee
compare to \r{RMC_1} and\r{RMC4.2}.

\begin{figure}[!ht]
\includegraphics[page=3,scale=1]{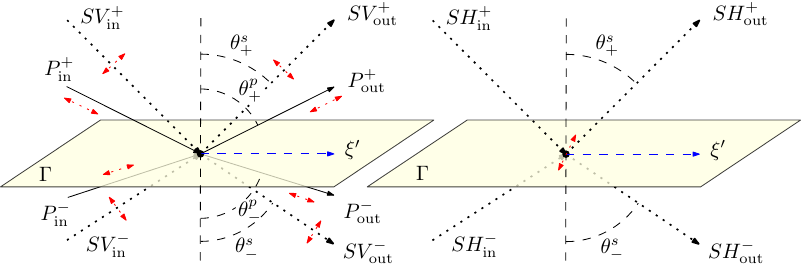}
\caption{The transmission problem in the (HH) case: the general case of eight waves with wave front set projected to the same covector. The SH waves behave as acoustic ones. 
}\label{pic2a}
\end{figure}

The properties of the SH components are similar to those of acoustic waves at an interfaces, see  \r{ac20} there, and the discussion following it. In particular, there is no mode conversion (on principal symbol level at least, which we study). 

As above, we can derive the following energy equality: 
\be{RMC_en2}
\begin{split}
&\rho_+ \cot\theta_+^p \left(|\PP_\text{out}^+|^2- |\PP_\text{in}^+|^2\right) +
\rho_+ \cot\theta_+^s \left( |\SV_\text{out}^+|^2 - |\SV_\text{in}^+|^2\right)\\  
&\quad +\rho_- \cot\theta_-^p   \left(|\PP_\text{out}^-|^2 -|\PP_\text{in}^-  |^2\right) + 
 \rho_- \cot\theta_-^s  \left( |\SV_\text{out}^-|^2- |\SV_\text{in}^-|^2\right)=0.
\end{split}
\ee
For future reference in the case of evanescent modes, we write \r{RMC_en2} as
\be{RMC_en3}
\begin{split}
&\Re\Big(\rho_+ \xi_{3,+}^p \left(|\PP_\text{out}^+|^2- |\PP_\text{in}^+|^2\right) +
\rho_+ \xi_{3,+}^s  \left( |\SV_\text{out}^+|^2 - |\SV_\text{in}^+|^2\right) \\  
&\quad +\rho_- \xi_{3,-}^p   \left(|\PP_\text{out}^-|^2 -|\PP_\text{in}^-  |^2\right) + 
 \rho_- \xi_{3,-}^s  \left( |\SV_\text{out}^-|^2- |\SV_\text{in}^-|^2\right)\Big) =0,
\end{split}
\ee
see \r{RMC_cot}. Written this way, \r{RMC_en3} holds even if the quantities above are not necessarily real; and the proof requires to multiply the first row of \r{main_system} by the \textit{conjugate} of the third one and the same for the second and the fourth ones. This is an energy identity, see the paragraph following \r{RMC_en0}. It says that the combined energy flux of all incoming waves on $\Gamma$ (on principal level) equals that of the outgoing ones. 

\begin{lemma}\label{lemma_A1}
The matrices $(A_\textrm{\rm in}^+, A_\textrm{\rm out}^+)$, $(A_\textrm{\rm in}^-, A_\textrm{\rm out}^-)$, $(A_\textrm{\rm in}^+, A_\textrm{\rm in}^-)$, $(A_\textrm{\rm out}^+, A_\textrm{\rm out}^-)$ are elliptic.  Also, system \r{2x2} is elliptic for $(\SH_\textrm{\rm in}^+, \SH_\textrm{\rm out}^+)$, and also for  $(\SH_\textrm{\rm in}^+, \SH_\textrm{\rm in}^-)$. 
\end{lemma}

\begin{proof}
The ellipticity of the first two follows from Lemma~\ref{lemma_A}. The ellipticity of the next two follows from the energy equality \r{RMC_en2}. The second statement follows from the fact that the corresponding determinants are negative, and positive, respectively. 
\end{proof}

Note that the ellipticity of $(A_\textrm{\rm in}^+, A_\textrm{\rm out}^+)$ and $(A_\textrm{\rm in}^-, A_\textrm{\rm out}^-)$ holds in the mixed and in the elliptic case as well by the proof of Lemma~\ref{lemma_A}. 

This has the following implications (without the claim that none of the amplitudes vanishes so far); compare with the discussion following \r{ac20}. Recall that we assume that the Cauchy data on the boundary is in the hyperbolic region with respect to all four speeds. 

\begin{itemize}
\item[(ii)] For every choice of the four incoming waves, there is a unique solution (ellipticity) for the four outgoing ones. Indeed, \r{RMC_en3} implies unique solution of the homogeneous problem.  
\item[(ii)] An incoming P wave (without any other incoming waves on either side) creates reflected P  and S waves and transmitted P  and S waves.  
\item[(iii)] The same is true for an incoming S wave.
\item[(iv)] [Control] For every choice of a principal amplitude of an outgoing transmitted P wave, one can choose incoming S  and P waves which would give that pre-assigned transmitted  P wave and no (on the principal level) transmitted S wave. The same is true for incoming P waves. 
\end{itemize}
 
\subsubsection{The hyperbolic-mixed (HM) case} \label{sec_hm}
Assume the wave front set of the Cauchy data is in the mixed  region in $\Omega_-$ but still in the hyperbolic one in $\Omega_+$. Since we work in the elliptic region for $c_{p,-}$, we will call the principal amplitude of the corresponding microlocal solution $P_-$ (no in/out), see Figure~\ref{pic_hm}.  

\begin{figure}[!ht]
\includegraphics[page=10,scale=1]{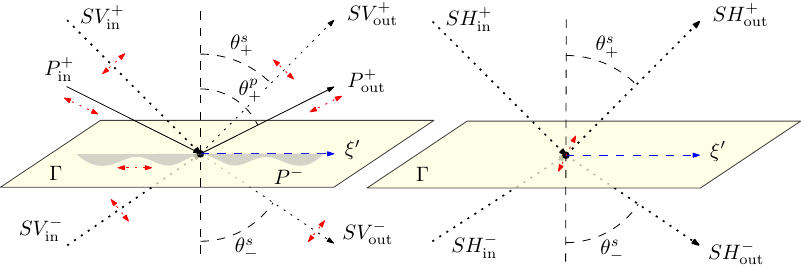}
\caption{The transmission problem in the hyperbolic-mixed (HM) case: The $P^-$ wave is evanescent, no incoming/outgoing parts. SH waves do not create  P waves. 
}\label{pic_hm}
\end{figure}

The approach we follow is the same as above --- we want to analyze the system \r{RMC4.1n}, and for example solve it for all outgoing waves given the incoming ones by proving ellipticity. What changes is that  $\xi_{3,-}^p$ becomes pure imaginary, see \r{xi_3pi}. One should also change the sign of $\xi_{3,-}^p$ in $A_\text{out}$ since there are not plus/minus square roots but those entries will be multiplied by zero below. 
 Then \r{main_system} reduces to 
\be{main_system2}
A_\textrm{in}^+(\PP_\text{in}^+ , \SV_\text{in}^+)^T + A_\textrm{out}^+(\PP_\text{out}^+, \SV_\text{out}^+)^T= 
A_\textrm{in}^-(0 , \SV_\text{in}^-)^T + A_\textrm{out}^-(\PP^-, \SV_\text{out}^-)^T,
\ee
see \r{u_C4'}. 
The energy equality \r{RMC_en3} reduces to 
\be{RMC_en4}
\rho_+ \xi_{3,+}^p \left(|\PP_\text{out}^+|^2- |\PP_\text{in}^+|^2\right) +
\rho_+ \xi_{3,+}^s  \left( |\SV_\text{out}^+|^2 - |\SV_\text{in}^+|^2\right) + 
 \rho_- \xi_{3,-}^s  \left( |\SV_\text{out}^-|^2- |\SV_\text{in}^-|^2\right)=0, 
\ee
see also \r{ac_fir3}. 
We get that for any choice of the three incoming waves, the resulting system for the three outgoing ones plus $P^-$ is elliptic. Indeed, it is enough to show this for the homogeneous system. If all incoming waves vanish, then 
\r{RMC_en4} implies $P_\text{out}^+ = SV_\text{out}^+=  SV_\text{out}^-=0$. Then the only possible non-zero vector in \r{main_system2} is $P^-$ but then we can see directly that \r{main_system2}  implies $P^-=0$. System \r{2x2} about the SH waves is unaffected by the ellipticity of the $P$ wave. 
Therefore, constructing the outgoing solution is a well-posed (elliptic) problem.

 As far as control from each side is concerned, on the negative one, where $P^-$ lies, the Cauchy data is structured; then so is on the positive side.  Therefore, the configuration on the positive side  cannot be controlled from the negative one. On the other hand, we can create any hyperbolic configuration on the negative side with appropriate waves on the positive one. In particular, if we want $P^-=0$, $SV^-_\text{in}=0$ and $SV^-_\text{out}\not=0$, we can take the Cauchy data of it and  solve \r{main_system2} for the plus amplitudes since on the positive side, we are in the hyperbolic region and the Cauchy problem is elliptic. 

Control for SH waves on principal level is the same as in the acoustic case since those waves do not create reflected/transmitted P or SV waves. Since we defined SH/SV waves on principal level only, and the system for the amplitudes is decoupled only on a principal level a priori, the control question needs a further clarification when  evanescent P  modes are possible. Let us say that we want to create S waves on the negative side with given principal amplitudes $\SV_\text{in}^-$, $\SV_\text{out}^-$, $\SH_\text{in}^-$, $\SH_\text{out}^-$, $P^-$. The argument above says that we can chose the principal amplitudes of the waves on the top to make this happen on a principal symbol level, see Figure~\ref{pic_hm}. Then we fix the six  waves on the positive side which have those amplitudes as their full ones in those coordinates. For each one, we need to solve, up to infinite order, a transmission, not a control problem, which is well posed. This would possibly create lower order waves on the negative side but it will not change the principal parts. In particular, if we want $\SV_\text{in}^- =\SV_\text{out}^- = \SH_\text{in}^-$ but  $\SH_\text{out}^-\not=0$, this step could create lower order $\SV_\text{in}^-$, $\SV_\text{out}^-$, $\SH_\text{in}^-$ waves. This is not a problem since we will need the principal parts later only. We apply the same argument in the cases below. 

\subsubsection{The mixed-mixed (MM) case} \label{sec_t_mm}
Then  on both sides, the S waves are hyperbolic, and $P^-$ and $P^+$ are evanescent, see Figure~\ref{pic_mm}.  In this case,  $\xi_{3,\pm}^p$ are pure imaginary, see \r{xi_3pi}. Then there is only one evanescent P wave 
in $\Omega_-$ and one in $\Omega_+$ and we omit the subscripts ``in/out'' for them.

As above, we  show below that on a principal symbol level, the energy is carried by the  S waves only. We also check directly that the homogeneous problem (no incoming waves) has the trivial solution only, including trivial evanescent modes $P^-$ and $P^+$. Therefore, we still get a well-posed problem for the outgoing solution. 

In \r{main_system2}, we can formally set $P_\text{in}^+=0$, $P_\text{out}^+=P^+$ and in the energy equality \r{RMC_en4}, we remove the $P$ amplitudes to get
\be{msmm}
A_\textrm{in}^+(0, \SV_\text{in}^+)^T + A_\textrm{out}^+(\PP^+ , \SV_\text{out}^+)^T= 
A_\textrm{in}^-(0 , \SV_\text{in}^-)^T + A_\textrm{out}^-(\PP^-, \SV_\text{out}^-)^T,
\ee
and
\be{RMC_en4mm}
\rho_+ \xi_{3,+}^s  \left( |\SV_\text{out}^+|^2 - |\SV_\text{in}^+|^2\right) + 
 \rho_- \xi_{3,-}^s  \left( |\SV_\text{out}^-|^2- |\SV_\text{in}^-|^2\right)=0, 
\ee
with $\xi_{3,\pm}^p$  pure imaginary as in \r{xi_3pi}. We will show that \r{msmm} is elliptic for $\SV_\text{out}^-$, $\SV_\text{out}^+$, $P^-$, $P^-$, given $\SV_\text{in}^-$, $\SV_\text{in}^+$. As before, it is enough to show that the homogeneous system is uniquely solvable. This follows from Lemma~\ref{lemma_A} or Lemma~\ref{lemma_A1} which remain true in the elliptic and the mixed regions. 
\begin{figure}[!ht]
\includegraphics[page=11,scale=1]{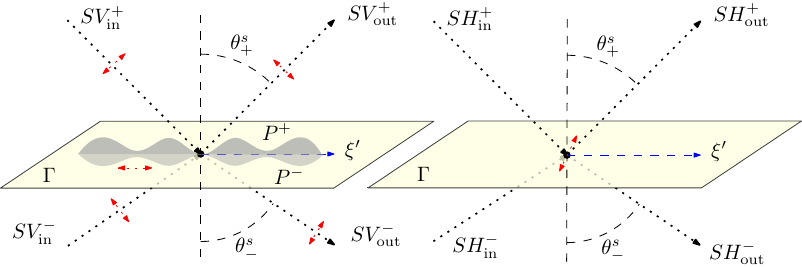}
\caption{The transmission problem in the mixed-mixed (MM) case: The $P^-$ and the $P^+$ waves are evanescent, no incoming/outgoing parts. The SH waves behave as acoustic ones. 
}\label{pic_mm}
\end{figure}
The SH waves  behave  as in the acoustic case, see \r{ac20} and \r{ac_en1} and as in the (HM) case.  

Control is possible for the SH waves. Let us say that we want to create SH waves on the negative side with prescribed principal amplitudes $\SH_\text{in}^-$, $\SH_\text{out}^-$ and no other waves there. On principal level, we choose $\SH_\text{in}^+$, $\SH_\text{out}^+$ to achieve that. Then, as above, we chose such S waves on the positive side with those principal amplitudes. Solving the direct transmission problem with (hyperbolic only) sources on the positive side, we may get  additional waves on the negative ones as shown on Figure~\ref{pic_mm}, left,  but they are lower order. 

One can also show that control for SV waves on either side is possible from the other one, which would create evanescent $P^+$ and $P^-$ modes as well. Indeed, to show that given any $SV^-_\text{in}$, $SV^-_\text{out}$, we can choose $SV^+_\text{in}$, $SV^+_\text{out}$ creating those waves plus the ``byproducts'' $P^-$, $P^+$, we need to show that \r{msmm} is elliptic for $SV^+_\text{in}$, $SV^+_\text{out}$, $P^-$, $P^+$. A direct but tedious computation shows that the determinant of this system equals
\[
-2\xi_{3,+}^s\mu_+\tau^2c_{s,+}^{-2}  c_{s,+}^{-4}  \left(  2\xi_1^2(\mu_+-\mu_-)(\xi_{3,+}^p+\xi_{3,-}^p)c_{s,-}^{2}c_{s,+}^{2} +\tau^2(\xi_{3,+}^p\mu_-  c_{s,+}^{2}-  c_{s,-}^{2} \mu_+\xi_{3,-}^p ) \right).
\]
The algebraic structure of this expression implies that this determinant is not identically zero for all $\xi_1$ unless none of the coefficients jump at the interface, and we assumed that this could not happen. Therefore, it could be zero for a discrete set of $\xi_1$'s only and then we have control.

\subsubsection{The hyperbolic-elliptic (HE) case} \label{sec_t_he}
Assume that both the P  and the S  waves on the negative side are evanescent but they are hyperbolic on the plus side, see Figure~\ref{pic_he}. Then we have full reflection on the positive side with respect to all waves. System \r{main_system2} reduces to
\be{ms32}
A_\textrm{in}^+(\PP_\text{in}^+ , \SV_\text{in}^+)^T + A_\textrm{out}^+(\PP_\text{out}^+, \SV_\text{out}^+)^T=  A_\textrm{out}^-(\PP^-, \SV^-)^T,
\ee
where $P^-$ and $SV^-$ are evanescent and $\xi_{3,-}^p$ and $\xi_{3,-}^s$ are pure imaginary as in \r{xi_3pi} and \r{xi_3si}. The energy equality takes the form
\be{en_mm}
\rho_+ \xi_{3,+}^p \left(|\PP_\text{out}^+|^2- |\PP_\text{in}^+|^2\right) +
\rho_+ \xi_{3,+}^s  \left( |\SV_\text{out}^+|^2 - |\SV_\text{in}^+|^2\right)  =0.
\ee
\begin{figure}[!ht]
\includegraphics[page=12,scale=1]{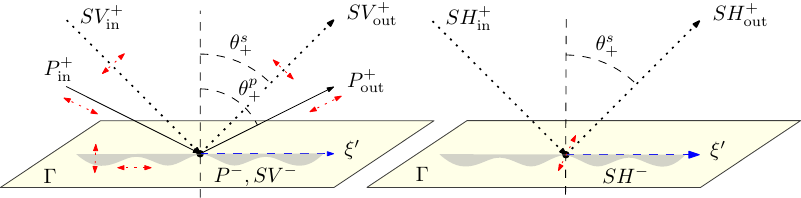}
\caption{The transmission problem in the hyperbolic-elliptic (HE) case: The $P^-$ and the $S^-$ waves are evanescent, no incoming/outgoing parts. The SH waves behave as acoustic ones with a total reflection on the top.} \label{pic_he}
\end{figure}
This is similar to the hyperbolic case in the Cauchy boundary value problem, see \r{sec_el_bvpC}. System \r{ms32} is elliptic for solving for $\SV_\text{out}^-$, $\SV_\text{out}^+$, $P^-$, $\SV^-$ by \r{en_mm} and Lemma~\ref{lemma_A1}. 
The SH waves are treated similarly. They experience a full reflection as in the acoustic case. 


\subsubsection{The mixed-elliptic (ME) case} \label{sec_t_me}
Assume that only the $\SV^+$ waves are hyperbolic. Then we have full reflection of the S wave on the positive side with transmitted evanescent $P^-$ and $S^-$ waves and mode converted $P^+$ one on the positive side, see Figure~\ref{pic_me}. System \r{main_system2} reduces to
\be{me32}
A_\textrm{in}^+(0, \SV_\text{in}^+)^T + A_\textrm{out}^+(\PP^+, \SV_\text{out}^+)^T=  A_\textrm{out}^-(\PP^-, \SV^-)^T,
\ee
where $P^-$, $\SV^-$ and $\PP^+$ are evanescent and $\xi_{3,\pm }^p$ and $\xi_{3,-}^s$ are pure imaginary. The energy equality takes the form
\be{en_me}
\rho_+ \xi_{3,+}^s  \left( |\SV_\text{out}^+|^2 - |\SV_\text{in}^+|^2\right)  =0.
\ee
\begin{figure}[!ht]
\includegraphics[page=13,scale=1]{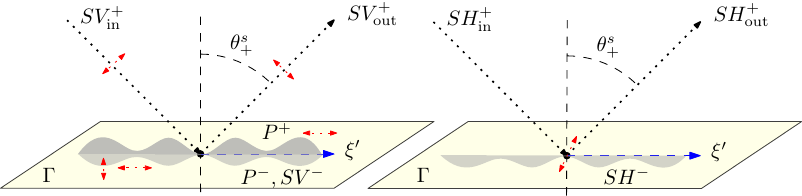}
\caption{The transmission problem in the mixed-elliptic (ME) case: Only the $\SV^+$ waves are hyperbolic. The SH waves behave as in the acoustic case and as in the (HE) case.} \label{pic_me}
\end{figure}
System \r{me32} is elliptic for solving for   $\SV_\text{out}^+$, $P^-$, $\SV^-$ by \r{en_me} and Lemma~\ref{lemma_A1}. 
The SH waves are treated similarly. They experience a full reflection as in the acoustic case. 


\subsubsection{The elliptic-elliptic (EE) case. Stoneley waves} \label{sec_t_ee}
We assume now that all waves on both sides are evanescent. Such solutions cannot be created by S or P waves hitting $\Gamma$ but they could be created by boundary sources. We will sketch the construction of such solution known as Stoneley waves first described by R.~Stoneley  \cite{Stoneley} in 1924  in case of flat boundary and constant coefficients, see also \cite{Yamamoto_elastic_89} for a curved boundary and constant coefficients.

We call the evanescent  amplitudes $P^-$, $P^+$, $SV^-$, $SV^+$. 
Then 
\be{ms_ee}
 A_\textrm{out}^+(\PP^+, \SV^+)^T=  A_\textrm{out}^-(\PP^-, \SV^-)^T.
\ee
Since $\xi_{3,-}^p$, $\xi_{3,-}^s$, $\xi_{3,+}^p$ and $\xi_{3,+}^s$ are all pure imaginary, with a positive imaginary part, the matrices above do not really have outgoing properties and the subscript ``out'' could be omitted. In this region, 
\be{u_C5x}
A_\text{out}^\pm := \begin{pmatrix}
\xi_1&\xi_{3,\pm}^s \\
-\xi_{3,\pm}^p&\xi_1 \\
-2\mu\xi_{3,\pm}^p \xi_1 & \mu(2\xi_1^2-c_{s}^{-2}\tau^2)  \\
 - \mu( 2\xi_1^2-c_{s}^{-2}\tau^2 )  &   -2\mu  \xi_{3,\pm}^s\xi_1 
\end{pmatrix},
\ee
see \r{u_C5}.  Then $F:= \det(A_\text{out}^+,- A_\text{out}^-)$ is a positively homogeneous function of $(\tau,\xi_1)$ of order $6$. Writing $F$ as $\xi_1^6$ times a function $F_0$ of $s:= |\tau|/\xi_1$ (and the base point $x'$), we get that $(A_\text{out}^+,- A_\text{out}^-)$ is elliptic (again, after adjusting the order of the last two rows from 2 to 1) where $F_0(s,x')\not=0$. Passing to an invariant formulation as in Section~\ref{sec_SV-SH}, we can replace $\xi_1$ by $|\xi'|$; then $s= |\tau|/|\xi'|$ with the norm of $\xi'$ being the covector one w.r.t.\ the metric $g$, which in the isotropic case is the boundary metric induced by the Euclidean one. Then $F$ is a homogeneous symbol. Assume that $F_0$ has a simple zero for some $s=c_\text{St}$ corresponding to the elliptic-elliptic region, i.e., in $s<\min(c_{s,-},c_{s,+})$. Then $F=\left(\tau^2-c_\text{St}^2|\xi'|^2\right)\!\tilde F$ with $\tilde F$ elliptic near $\Sigma_\text{St}:= \{\tau^2=c_\text{St}^2|\xi'|^2\}$. Then $(A_\text{out}^+,- A_\text{out}^-)$ is a \PDO\ of real principal type (again, the order can be adjusted to be one for all rows) in the sense of \cite{Dencker_polar}. Singularities on $T^*\Gamma$ propagate along the null bicharacteristics of the Hamiltonian $H:= \left(\tau^2-c_\text{St}^2|\xi'|^2\right)$. This is a wave type of Hamiltonian with a wave speed $c_\text{St}$ which is slower that the S and the P speeds on either part of $\Gamma$. A well posed problem would be, for example, one with Cauchy data on $\{t=0\}\times\Gamma$. 

For every microlocal solution on $\R_t\times\Gamma$, we can use its Dirichlet data to extend it to a microlocal solution on both sides of $\Gamma$ as in Section~\ref{sec_6.3}, see also Section~\ref{sec_Rayleigh}. Rayleigh waves,  can be considered as a limit case of Stoneley waves.  

The function $F_0$ does have (simple) zero in some cases, at least. Some examples can be found in Stoneley's original  paper \cite{Stoneley}. 

\subsection{Summary}\label{sec_Tr_summary}
We  summarize some of the results above as follows. 
\begin{itemize}
	\item[(HH)] the \textbf{hyperbolic-hyperbolic} case:  we have both P and S waves on either side; each incoming wave creates two reflected and two transmitted (refracted) ones, with mode conversion. 

	\item[(HM)] The \textbf{hyperbolic-mixed} case: on one side there are both P and S waves, on the other one, only S waves exists (as solutions propagating singularities);  the P wave is evanescent. On other hand, there is  total internal reflection of P waves but they can still create transmitted S waves.

	\item[(HE)] The \textbf{hyperbolic-elliptic} case: the S and the P waves on one side are hyperbolic;  the  S and the P waves on other side are evanescent. Then there is a full reflection from the first side, and the transmitted waves are only evanescent. 

	\item[(MM)] The \textbf{mixed-mixed} case: the S waves on both sides are hyperbolic (propagate singularities);  the P waves on both sides are evanescent. In particular, an incoming S wave reflects and refracts; and it creates two evanescent P waves on either side by mode conversion. 

	\item[(ME)] The \textbf{mixed-elliptic} case: Only the S wave on one side is hyperbolic. In particular, an incoming S wave reflects; and it creates two evanescent P waves on either side by mode conversion and one ``reflected'' P evanescent one. 

	\item[(EE)] The \textbf{elliptic-elliptic} case: All waves are evanescent. Such waves cannot be created by a P or an S wave hitting $\Gamma$ but it could be created by a boundary source. The transmission problem may lose ellipticity and allow for solutions (Stoneley waves) concentrated near $\Gamma$. 
\end{itemize}

\subsection{Justification of the parametrix} 

In the construction above, we work with microlocal solutions which may have singularities but they, and their first derivatives have traces in timelike surfaces. We assume that solutions have wave front set disjoint from  bicharacteristics tangential to some of interfaces which can be achieved by choosing the wave front of their Cauchy data disjoint from projections of such directions in $T^*\Gamma$. The later set has a zero measure on $S^*\bo$ for $t$ restricted to any fixed finite interval. The construction actually provides an FIO,  mapping $f$ to the microlocal outgoing solution $u$ with that boundary data.  

To justify the parametrix, we need to subtract it from the actual solution and show that the difference is smooth up to each interface $\Gamma_i$. Such a difference $w$ would solve a non-homogeneous problem 
\be{1A}
\left\{\begin{array}{lr}
u_{tt} -Ew \in C^\infty(\R\times\bar \Omega),\\ 
  w|_{\R\times\bo}   \in C^\infty(\R\times \bo),\\
\ \! \! [w]|_{\Gamma_j} , [Nw]|_{\Gamma_j}    \in C^\infty(\R\times \Gamma_j), \quad j=1,\dots,k, \\
  w|_{t<0}=0. 
\end{array}\right.
\ee

A slightly weaker version of this claim can be proven, which is sufficient for our purposes. We claim that $w$ is $C^\infty$ away from $\R\times \Gamma_j$ and $\R\times\bo$, and indeed is conormal at these two in the precise sense that $w\in\Hbloc^{1,\infty}$, meaning $w$ and its first derivatives are in $L^2$ locally, and the same remains true if vector fields tangent to $\R\times \Gamma_j$ and $\R\times\bo$ are applied to these iteratively. While this is standard in the scalar case, a proof for (principally) scalar wave equations, for transmission problems, based on quadratic form considerations, showing regularity relative to the quadratic form domain, is given in \cite{DUV:Diffraction}. This proof uses b-pseudodifferential operators, introduced by Melrose \cite{Melrose:Transformation}, see also \cite{Melrose:Atiyah}, and \cite{DUV:Diffraction} for a brief summary. The simple observation made in \cite[Section~4]{DUV:Diffraction} is that when one has an internal hypersurface, such as $\R\times\Gamma_j$, one can treat it as a boundary for this b-analysis by using b-pseudodifferential operators on each half-space (which are manifolds with boundary) with matching normal operators at the common boundary; this was used in \cite[Section~4]{DUV:Diffraction} to prove propagation of singularities in the principally scalar setting. The elastic problem is not principally scalar, which indeed makes the proof of propagation of singularities significantly more difficult using these tools. However, the propagation of global regularity, in the sense that regularity, as measured by $\Hbloc^{1,m}$ (i.e.\ the space with $m$ b-, or tangential, derivatives relative to $H^1_\loc$), propagates from $t<0$ to $t\geq 0$ when the right hand side has regularity in $\Hbloc^{-1,m+1}$ (i.e.\ the space with $m+1$ b-, or tangential, derivatives relative to $H^{-1}_\loc$) is straightforward as it does not require microlocalization; slightly modified energy estimates work. This has been carried out in detail by Katsnelson for the elastic wave equation on manifolds with edges in \cite[Chapter~11]{Katsnelson:PhD}. The latter are actually more complicated than our setting as the domain of the operator is more delicate, and an essentially identical method of proof works in our case. We also refer to \cite{Katsnelson:Diffraction} for a brief summary.

We refer to   \cite{Yamamoto_09} as well, where boundary regularity in the case of constant parameters has been studied.


\section{The inverse problem}\label{section_GC}
Assume that there exist two  smooth non-positive functions $\foliation_s$ and $\foliation_p$ in $\Omega$ with $\d\foliation\not=0$, $\foliation^{-1}(0)=\bo$, and $\foliation^{-1}(-j)=\Gamma_j$, $j=1,\dots,k$ where $\foliation$ is either $\foliation_s$ or $\foliation_p$.  Assume that the level sets $\foliation_{s}^{-1}(c)$, $\foliation_{p}^{-1}(c)$ are strictly convex w.r.t.\ the  speed $c_s$, $c_p$, respectively, when viewed from $\Gamma_{0}=\bo$. 
Of course, we may have just one such function, i.e., $\foliation_s=\foliation_p$ is possible. 

Recall that  the foliation condition implies non-trapping  as noted in \cite{SUV_localrigidity}, for example. In our case, this means that rays in $\Omega_j$ not hitting $\Gamma_j$ would hit $\Gamma_{j-1}$ both in the future and in the past. 

\subsection{Recovery in the first layer $\Omega_1$} 
We show first that we can recover $c_p$ and $c_s$, and then $\rho$, in the first layer $\Omega_1$, i.e., between $\bo$ and $\Gamma_1$. In other words, if $\tilde \rho$, $\tilde \mu$, $\tilde \nu$ is another triple of coefficients which have the same piecewise smooth structure with jumps across some $\tilde \Gamma_j$,  producing the same DN map, then they coincide with the non-tilded ones. In the lemmas below, we need solutions with a single  incoming singularity (more precisely, with a single radial ray due to the conic nature of the wave front sets) which we can trace until its branches hit $\bo$ again. We can do this in two ways: first, we can have $f$ in \r{1} with such a single singularity but when we need a specific polarization, we can achieve that by choosing  the potential $w$ appropriately, with that singularity. Since the operators $U_\text{in}$ and $U_\text{out}$, see \r{Ub} are elliptic in all regions, then the boundary trace of the potentials would have the same wave front sets as the boundary trace of the solution $u$. Or, one can have  $\WF(f)$ in a small set by choosing   $\WF(w)$ on the boundary small enough and then pass to a limit when $\WF(f)$ shrinks to a single point. Since the arguments based on SH/SV waves require us to trace the leading  singularities, i.e., we want to have a well defined order,  working with singularities in a small conic set, for example conormal ones, is more convenient. We assume in this section that $g$ is Euclidean since we will need the results of Rachele \cite{Rachele_2000, Rachele03}, and Bhattacharyya \cite{Bhattacharyya_18},  see Remark~\ref{rem_Rachele}. 

\begin{lemma}\label{lemma_G1} 
Under the convex foliation assumption,   $\Lambda$, known for $T\gg1$ determines uniquely  $\Gamma_1$, $c_s$ and $c_p$ in $\Omega_1$. If, in addition, $c_p\not=2c_s$ pointwise in $\Omega_1$, then $\rho$ is uniquely determined in $\Omega_1$ as well. 
\end{lemma}

\begin{proof} 
In this and in the following proof, we consider another triple $\tilde \rho$, $\tilde \mu$, $\tilde \nu$ with the same $\Lambda$, and show that the corresponding quantities, in this case  $\Gamma_1$ and the three coefficients, coincide. Sometimes, we say that a certain quantity, for example $c_s$, is known or can be recovered in some region to indicate that $c_s=\tilde c_s$ there. 

First, by \cite{Rachele_2000}, we can recover the full jets of $\rho$, $c_p$ and $c_s$ on $\bo$. We will recover the speeds $c_s$ and $c_p$ first. 
This follows from \cite{SUV_localrigidity}, in any subdomain separated from $\Gamma_1$, i.e., for $-1+\eps\le \foliation\le0$, $\forall \eps\in(0,1)$, with $\foliation =\foliation_s$  or $\foliation =\foliation_p$, 
and it is also  H\"older stable there. Indeed, for every  unit P or S geodesic connecting boundary points and not intersecting $\Gamma_1$, we can construct a microlocal P or S solution in a small neighborhood of that geodesic, extended a bit outside $\Omega$ where the coefficients are extended smoothly as well. Let $f$ be the Dirichlet data of that solution on $\R_+\times\bo$.  Then the outgoing solution $\tilde u$ having the same Dirichlet data has the same Neumann data as well. Also, the solution will be a P or an S wave, respectively as well, since this property is determined by the trace of $c_p$ and $c_s$ on $\bo$, which we recovered, see the end of section~\ref{sec_ED1m}. 
 Therefore, singularities hitting $\bo$ from inside, will be the same (a singularity hitting $\bo$ must create singular Cauchy data by the analysis in Section~\ref{sec_Ac_Cauchy}). So the scattering relations related to $c_s$ and $c_p$ are the same as those of  $\tilde c_s$ and $\tilde c_p$ restricted to those geodesics. Note that this argument requires us to know that the corresponding geodesics for the second system do not hit $\tilde \Gamma_1$. If they do, we would get reflected waves of both kinds (with a possible  exception of specific angles which does not change the argument), and we would not get the same Cauchy data. Another way to exclude such rays is to note that they would create singularities of the lens relation near rays tangent to $\tilde \Gamma_1$. 

This proves that $\Omega_1\subset\tilde  \Omega_1$, i.e., $\tilde \Gamma_1$ is below $\Gamma_1$, and that $c_p=\tilde c_p$, $c_s=\tilde c_s$ in $\Omega_1$. On the other hand, we can swap $\tilde \Gamma_1$ and $\Gamma_1$ in this argument, therefore $\tilde \Gamma_1=\Gamma_1$. Then $c_s$ and $c_p$ are uniquely determined there. By \cite{Bhattacharyya_18}, one can recover $\rho$ in $\Omega_1$ as well under the stated condition, therefore then we can recover $\lambda$, $\mu$, too. 
\end{proof}

Note that here, and in what follows, we have precise control of $T$ which we do not make explicit. Also, local knowledge of $\Lambda$ up to a smoothing operator yields recovery in an appropriate domain of influence, see also \cite{SUV_elastic} for the case of smooth coefficients. 

\subsection{Recovery in the second layer $\Omega_2$}  
In the next lemmas, we show that we can recover the two speeds  in $\Omega_2$ under some conditions. The obstruction to the application of the method (but not necessarily to the uniqueness) is existence of totally reflected P and/or S rays on the interior side of $\Gamma_1$ for all times (or for long enough, for the case of data on a finite time interval). 
Since we need rays converging to tangential ones, the microlocal conditions can be described in terms of the sign of the jumps of the speeds at $\Gamma_1$.

In what follows, $c|_{\Gamma^\pm}$ denotes the limit of $c(x)$ as $x$ approaches $\Gamma$ from the exterior/interior. 

\begin{lemma}\label{lemma_G2}
Under the assumption in the first sentence of Lemma~\ref{lemma_G1}, assume additionally that 
\be{G3}
c_s|_{\Gamma_1^+}< c_s|_{\Gamma_1^-}. 
\ee
Then $\Gamma_2$ and $c_s$ are determined uniquely in (the uniquely determined) $\Omega_2$. 
\end{lemma}

 We can interpret \r{G3} as strict convexity of  $\Gamma_1$ w.r.t.\   $c_s$ with a jump since increasing the speeds with depth guarantees strict convexity of the level surfaces. It guarantees no total full reflection of S ways from $\Omega_2$ to $\Omega_1$. On the other hand, \r{G3} implies
\be{G5}
c_s|_{\Gamma_1^+}<c_s|_{\Gamma_1^-}  <   c_p|_{\Gamma_1^-} 
\ee
but the only thing we know about  $c_p|_{\Gamma_1^+}$ is that it  is greater than  $c_s|_{\Gamma_1^+}$. In particular, there could be evanescent S to P or P to S transmission from $\Omega_2$ to $\Omega_1$; or they all could be hyperbolic.

 \begin{figure}[!ht]
\includegraphics[page=9,scale=1]{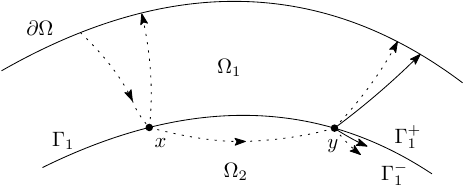}
\caption{Solid curves are P waves. Dotted curves are S waves. We can create an  SH wave  connecting points on $\Gamma_1$ and no other waves from or to $x$ below $\Gamma_1$ by choosing it to be SH on $\Gamma_1$ near $x$. The reflected and/or the transmitted P waves at $x$ and $y$ could be evanescent.
}\label{pic_IP1}
\end{figure}
 
\begin{proof}[Proof of Lemma~\ref{lemma_G2}] 
Let $x$, $y$ be on $\Gamma_1$ connected by a unit speed S geodesic $\gamma_0$ staying between $\Gamma_1$ and $\Gamma_2$. 
We take an outgoing  microlocal solution $u$ concentrated near $\gamma_0$, so that $u$ is singular near $x$ when $t$ is near $t_1$; and $t=t_2$ corresponds to $y$. 
We choose $u$ to be an SH wave on $\Gamma_1$ near  $x\in \Gamma_1$, see Figure~\ref{pic_IP1}.  
The SH waves behave as acoustic ones on both sides in the (HH), (MH), (HM) and the (MM) cases on principal symbol level, and all those cases are possible. Recall that our convention is to list the top  first; in particular, the (MH) configuration is  the (HM) one in Section~\ref{sec_hm} with the top and the bottom swapped.  To create such a wave, we just need to take an S wave coming from $\bo$ so that its trace on $\Gamma_1$ is SH; this can be done by time reversal. On principal level, there will be no other singularities below $\Gamma_1$ until that wave hits $\Gamma_1$ again. 
Then that solution will create singular Cauchy data near $y$ and $t$ near $t_2$. It is an S wave but not necessarily  an SH one at $y$.  At least one of the two waves transmitted back to $\Omega_1$ would have non-zero principal amplitude if there are two hyperbolic ones, or if there is an S one only, it would be non-zero by the results of the previous section. Then there will be at least one singularity hitting $\bo$ (which we allow to leave $\Omega$, as above). On the other hand, there might be other waves hitting $\bo$ at the same place and time coming from waves at $y$ below $\Gamma_1$ which can reflect of refract. 
Since we allow all those waves to leave $\Omega$ freely, they would have different wave front sets or polarizations, and in particular they cannot cancel or alter the singularity of the Cauchy data generated by  the waves coming directly from $y$. The simplest way to see that is to do time reversal from the exterior of $\Omega$ back to $\Omega$. 

The speeds $c_s$ and $c_p$ are the same for both systems in $\Omega_1$ by Lemma~\ref{lemma_G1}. 
We can assume $t_1\gg1$ so that $u$ is smooth on $\bo$ for $t<\eps$ for some $\eps>0$. 
Since the solutions constructed above for both systems have the same Cauchy data on $(0,T)\times \bo$ and we can choose $T\gg1$, we conclude that the principal part of $u$ on $\Gamma_1$ near $t=t_1$ is uniquely reconstructed.  Note that this argument does not require recovery of $\rho$ in $\Omega_1$ since we only need the principal amplitudes and  by \cite{Rachele_2000}, they do not depend on $\rho$. 
There might be other singularities on $\Gamma_1$ but we can identify $y$ as the first point a singularity comes back to $\Gamma_1$, and we can determine the travel time through $\Omega_2$ as well. 
Taking $y\to x$, we can recover the full jet of $c_s$ on $\Gamma_1^+$   by \cite[Lemma~2.1]{SUV_localrigidity}. 
Since we now know the S metric on $\Gamma_1^-$, this is enough to recover the lens relation related to $c_s$  on $\Gamma_1^-$, restricted to rays  not hitting $\Gamma_2$.  By \cite{SUV_localrigidity}, this determines $c_s$ in $\Omega_2$ uniquely, i.e., $c_s= \tilde c_s$ in $\Omega_2\cap\tilde\Omega_2$. 

We remark that the magnitudes of the refracted SH waves into $\Omega_2$ at $x$ may vary for each of the two systems since we do not know $\rho_-:=\rho|_{\Gamma_1^-}$; see \r{RMC4.2} where we can write $\mu_\pm = \rho_\pm c_s^2$. Their directions however do not depend on $\rho_-$ and each one can vanish only for a specific incidence angle (a priori different for each system). 

Finally, $\Gamma_2=\tilde\Gamma_2$ since the presence of the interface $\Gamma_2$ would create a singularity of the lens relation of the reflected S wave (plus a possible P wave); which would be detected om $\Gamma_1$. 
\end{proof}

\begin{lemma}\label{lemma_G3}
Under the assumptions of Lemma~\ref{lemma_G2}, assume in addition  
\be{G6}
c_p|_{\Gamma_1^+}< c_p|_{\Gamma_1^-} . 
\ee
Then $c_p$ is uniquely determined in $\Omega_2$. 
\end{lemma}

\begin{remark}
Conditions \r{G5} and \r{G6} say that there is no  total internal reflection of $P\to P$ and $S\to S$ rays from $\Omega_2$ to $\Omega_1$. On can still have evanescent transmitted $S\to P$ waves from the interior. More precisely, we have the following two generic cases (excluding $c_s|_{\Gamma_1^-}  =  c_p|_{\Gamma_1^+}$): 
\be{G8}
c_s|_{\Gamma_1^+}<c_s|_{\Gamma_1^-}  <  c_p|_{\Gamma_1^+} < c_p|_{\Gamma_1^-}, 
\ee
and
\be{G9}
c_s|_{\Gamma_1^+} <  c_p|_{\Gamma_1^+} <c_s|_{\Gamma_1^-} < c_p|_{\Gamma_1^-}, %
\ee
see also \r{G5}. Evanescent $S\to P$ transmission from the interior happens when \r{G8} holds. This is not a problem for the proof since we recovered $c_s$ in $\Omega_2$ using $SH$ waves. On the other hand, \r{G9} implies that all rays from the interior create transmitted rays, i.e., the wave front on $\Gamma_1^+$ is in the hyperbolic region.

\end{remark}

\begin{proof}[Proof of Lemma~\ref{lemma_G3}] 
We want  to use $P$ rays in $\Omega_2$ not hitting  $\Gamma_2$, in particular  having Cauchy data on $\Gamma_1^-$ in the hyperbolic region $c_p|_{\Gamma_1^-}|\xi'|<|\tau|$. By \r{G6}, that Cauchy data will fall in the hyperbolic region on $\Gamma_1^+$; in other words, we have  the (HH) case.  We use the control argument in \cite{caday2019recovery} now. 
In the (HH) case, near $x$ and $t=t_1$, we can create an  outgoing P wave in $\Omega_2$ with no other S or P waves there; in other words, in Figure~\ref{pic2}, only $P^-_\text{out}\not=0$ among all waves on the bottom. Then we extend the four waves on the top until they leave $\Omega$. At $y$, where that wave hits $\Gamma_1$ again near $t=t_2$, we can apply the same argument to make sure that there are no reflected rays in $\Omega_2$, see Figure~\ref{pic_IP2}. 
 \begin{figure}[!ht]
\includegraphics[page=15,scale=1]{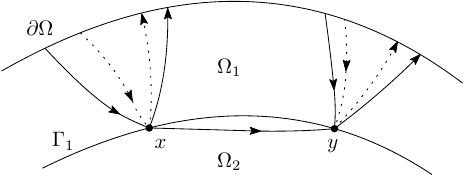}
\caption{Solid curves are P waves. Dotted curves are S waves. We can create a P wave  connecting points on $\Gamma_1$ and no other waves from or to $x$ below $\Gamma_1$ by choosing carefully the sources on the  top. At $y$, we can make sure that there are no reflected waves by choosing the sources on the top as well.  
}\label{pic_IP2}
\end{figure}
By energy preservation, we cannot have zero principal amplitudes of all four rays above $y$. Then by time reversal from $\bo$, we would know that there is a singularity on $\Gamma_1$ at $y$ and $t=t_2$, and we would know its wave front set. Note that we do not require knowledge of $\rho$ in $\Omega_1$ and $\Omega_2$. In principle, the second (tilded) system may have an S wave starting from $x$ at $t=t_1$. By the paragraph following Remark~\ref{rem_Rachele}, we must have a non-trivial P wave near $x$ and $y$ (since we have recovered $c_s$ already). The P wave arriving at $y$ at $t=t_2$ might a priori be due to an S wave from $x$ in $\Omega_2$ which has reflected at $\Gamma_1$ and mode converted by this is not possible because this would have created a singularity at a moment in the interval $(t_1,t_2)$ but we know that such singularity does not exist for either system. Therefore, this recovers the P travel time from $x$ to $y$. 
Then we recover $c_p$ in $\Omega_2$ as in the proof of the previous lemma. 
\end{proof}

Combining those two lemmas, we get the following.

\begin{theorem}\label{thm_ip} 
Assume we have two triples of coefficients  $ \rho$, $ \mu$, $ \nu$ and  $\tilde \rho$, $\tilde \mu$, $\tilde \nu$; and  $\Lambda=\tilde\Lambda$ with $T\gg1$. Assume the foliation condition and  \r{G3} and \r{G6} for each one of them.  
Then  $\Gamma_1=\tilde \Gamma_1$, $\Gamma_2=\tilde \Gamma_2$ and $c_s=\tilde c_s$,  $c_p=\tilde c_p$ in $\Omega_1\cup \Omega_2=\tilde \Omega_1\cup\tilde\Omega_2$. Also,  if $c_p\not=2c_s$ in $\Omega$, then $\rho=\tilde\rho$ in $\Omega_1$. 
\end{theorem}

\subsection{Recovery of the speeds in the third, etc., layers} 
This construction can be extended by induction under appropriate conditions:

\begin{theorem}\label{thm_ip2} 
Assume we have two triples of coefficients  $ \rho$, $ \mu$, $ \nu$ and  $\tilde \rho$, $\tilde \mu$, $\tilde \nu$. Let 
\be{G3'}
c_s|_{\Gamma_j^+}< c_s|_{\Gamma_j^-}, \quad c_p|_{\Gamma_j^+}< c_p|_{\Gamma_j^-} ,  \quad   j=1,\dots,k.
\ee
Assume the foliation condition in   $\Omega_1\cup\dots\cup\Omega_k$.  
 If    $\Lambda=\tilde\Lambda$ with $T\gg1$, then   $\Gamma_j=\tilde \Gamma_j$, $j=1,\dots,k$ and $c_s=\tilde c_s$,  $c_p=\tilde c_p$ in $\Omega_1\cup\dots\cup\Omega_k$.
\end{theorem}

\begin{proof}
We  show that one can recover the two speeds in $\Omega_3$ and then the theorem  follows by induction.

 \begin{figure}[!ht]
\includegraphics[page=20,scale=1]{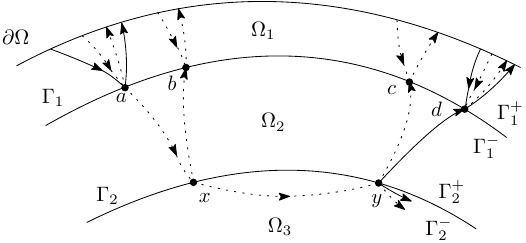}
\caption{Solid curves are P waves. Dotted curves are S waves.   We can create an S wave  connecting points $x$ and $y$ on $\Gamma_2$ so that it is an SH wave at $x$.
}\label{pic_IP3}
\end{figure}

We show that we can recover $c_s$ there first following the proof of  Lemma~\ref{lemma_G2}. 
Fix two points $x$ and $y$ on $\Gamma_2$. We keep them close enough so that the S geodesic connecting them does not touch $\Gamma_3$ (if there is $\Gamma_3$, i.e., if $k\ge3$). 
We construct a  solution $u$ below $\Gamma_2$, between $x$ and $y$, of S  type (at principal level, as everywhere in this section), see Figure~\ref{pic_IP3}.  We chose the solution to be SH at $x$ but this is not essential. At $y$, there might be reflection, transmission and mode conversion to evanescent modes. 
 Then near $x$ and at $y$ (and the corresponding times $t_1$ and $t_2$), the traces of this solution on $\Gamma_2$ is in the (HH), (HM) or the (MM) region by \r{G3'}, with the possible exception of finitely many angles giving rise to tangential rays. On the other hand, on principal level, there are  only incoming and reflected S waves at $\Gamma_2^+$ near $(t_1,x)$, satisfying the transmission conditions, and we can arrange no incoming waves at $x$ from $\Omega_3$ by the control argument for SH waves. 

We extend those microlocal solutions to $\Omega_2$ and $\Omega_1$ first as in the proof of Lemma~\ref{lemma_G2}. We do this starting from $x$ first. 
On $\Gamma_1$, each of the two S branches (meeting at $x$) are in one of the three  regions mentioned above excluding directions of measure zero). In each one of those cases, we can choose four or two waves on the top, i.e., in $\Omega_1$ which cancel a reflected wave. At $\Gamma_1^-$, we decompose all S waves into SH and SV ones. The latter can be treated as acoustic waves and can be controlled from the top. The SV ones can be controlled as well as we showed in sections~\ref{sec_HH},  \ref{sec_hm} and  \ref{sec_t_mm}. In Figure~\ref{pic_IP3}, for example, point $a$ corresponds to either an (HH) or an (HM) case; and point $b$ corresponds to an (MM) case; so does point $c$. Then we extend all waves in $\Omega_1$ to the exterior of $\Omega$, i.e., we let them leave $\Omega$.


At the point $y$, we do the same for the S and the P wave propagating into $\Omega_2$. 
For the $p$ wave (hitting $\Gamma_1$ at $d$ in Figure~\ref{pic_IP3}), we are at the (HH) zone at $\Gamma_1$, and we apply  the control argument we used before.

The so constructed microlocal solution vanishes (in this context, that means that it has no leading order singularities) for $t\ll0$, and by a  shifting $t_1$, we may assume that this happens for, say, $t<\eps$ with some $\eps>0$ (we need $\eps$ so that we can do a smooth cutoff between $t=0$ and $t=\eps$ and construct an actual solution with the same singularities). Choose $T\gg1$ so that all outgoing branches starting from $x$ or $y$ reach $\bo$ before that time. 

We are in the situation of Lemma~\ref{lemma_G2} now with $\Gamma_1$ playing the role of $\bo$ there with one difference.  
We have not recovered $\rho$ in $\Omega_2$. We  claim however that that near the micolrocal solutions along the rays hitting $x$ and $y$, $\tilde u$ (corresponding to the second system) has singularities of the same order as $u$. This follows form the following: the Cauchy data on $\R\times \Gamma_1^+$ and that on $\R\times\Gamma_1^+$ are related by the transmission conditions \r{tr}. It follows by \r{Nu} that they are related by an elliptic operator depending on $\rho$ (recall that we view the three independent coefficients as being $c_s$, $c_p$ and $\rho$). That dependence makes the refracted and the reflected amplitudes $\rho$ dependent, but it does not change the property of their principal parts being non-zero (except for specific angles). The same conclusion could have been reached by examining thse qualitative behavior of the solution of the microlocal systems, say \r{RMC4.1n} and \r{2x2} in the (HH) case, as a function of $\rho_-$. Therefore, $\tilde u$ has  leading order singularities in $\Omega_2$ along the same rays as $u$ does. By the proof of Lemma~\ref{lemma_G2}, 
$y$ is uniquely recovered as the first time the S wave from $x$ hits $\Gamma_2$ again. Then the boundary distance function related to $c_s$ in $\Omega_3$ is uniquely recovered for $x$ close to $y$, which recovers the jet of $c_s$ at $\Gamma_2^-$. Then we know the $c_s$ lens relation as well, along rays not touching $\Gamma_3$. As in the proofs above, we can detect where $\Gamma_3$ is and also recover $c_s$ in $\Omega_3$.

The recovery of $c_p$ in $\Omega_3$ goes along the same lines as the proof of Lemma~\ref{lemma_G3}, using the arguments above. 
We create a single P wave below $\Gamma_2$ connecting $x$ and $y$ and extend it until it reaches $\bo$ at both sides.  At $\Gamma_2$, we are in the (HH) case, each ray, extended upwards, will create four new ones. On the upper surfaces, we can have any of the (HH), (HM) and the (MM) cases as above. 

We can recover $\Gamma_3$ (if exists) by as in the previous lemmas. 

The proof for $k>3$ follows by induction. 
\end{proof}


\begin{remark}
Recovery of $\rho$ in $\Omega_j$, $j\ge2$ seems delicate. The arguments in \cite{Rachele03, Bhattacharyya_18} require the knowledge of the jet of $\rho$ at the boundary up to order three, which is true on $\bo$ by \cite{Rachele_2000} but proving this on $\Gamma_j^-$, $j=1,2,\dots$ seems to be not easy. 
\end{remark}

\subsection{Exploiting mode conversion; the PREM model} 
In the results above, we needed to ensure no total internal reflection of S or P waves or both, from the interior. The mode conversion was not used to obtain information, it was rather a difficulty we had to overcome. Below we show how one can use mode conversion to recover $c_p$ when the P waves are totally reflected but the refracted $S$ wave to the exterior is not. 

In the Preliminary Reference Earth Model (PREM) \cite{DZIEWONSKI1981297},  in the Upper and the Lower Mantle, the S and the P speeds increase with depth, ``on average'', except on a small interval close to the surface. At the boundary of the Lower Mantle and the Outer Core however, the P speed jumps down with depth, hence it does not satisfy \r{G6} on that interface. The S speed jumps down to zero, i.e., the Outer Core is believed to be  liquid. This violates \r{G3} on that interface (and there are no S waves in the Outer Core anyway). Therefore, the P waves in the Outer Core close to tangent ones to their upper boundary are totally reflected (as P waves only) and the results above do not apply for the recovery of $c_p$. In this case we can use mode conversion however because PREM shows that those P waves actually produce transmitted (hyperbolic) S waves into the exterior, i.e., condition \r{G10} below holds. 

An analysis of a solid-liquid model is certainly possible with the methods we develop but it is beyond the scope of this work (see also \cite{de2015system}). We will sketch arguments based on the dynamical system only assuming no S waves below $\Gamma_1$ (formally, $c_s=0$ there). \textit{Those arguments are not a proof} since we assume preservation of the microlocal properties in the limit $c_s|_{\Omega_1}\to0$. 

Assume
\be{G10}
c_s|_{\Gamma_1^+}< c_p|_{\Gamma_1^-}. 
\ee

First, we can determine the two speeds in $\Omega_1$ per Lemma~\ref{lemma_G1}.

To recover $c_p$ in $\Omega_2$, take a P geodesic in $\Omega_2$ connecting $x$ and $y$ on $\Gamma_2$, so that it does not hit $\Gamma_2$; see Figure~\ref{pic_IP5}. We can construct a microlocal solution $u$  near it so that it is obtained by an S wave in $\Omega_1$ through mode conversion at $\Gamma_1$. 
 \begin{figure}[!ht]
\includegraphics[page=17,scale=1]{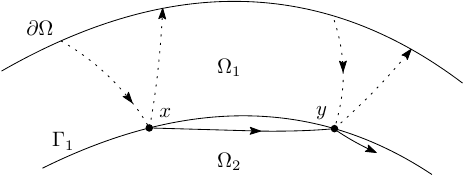}
\caption{Solid curves are P waves. Dotted curves are S waves. We can create a P wave in $\Omega_2$ connecting points on $\Gamma_1$ through mode conversion of an S wave coming from $\Omega_1$.
}\label{pic_IP5}
\end{figure}
To construct such an incoming solution, we can start with such between $x$ and $y$ and time reverse it. Then we take the S branch in $\Omega_1$, which on Figure~\ref{pic_IP5} is represented by the dashed most left incoming ray; and let it propagate. There will be a mode conversion in a a neighborhood of $x$, giving use the desired solution. It will have a non-zero principal level energy except possibly for directions of  measure zero. There might be a mode converted reflected P wave at $x$ back to $\Omega_1$, not shown on  Figure~\ref{pic_IP5}. If so, we let it propagate and exit $\Omega_1$, similarly to the reflected S wave. 
There will be a reflected P wave, and a transmitted S wave of non-zero principal energy except possibly for angles of measure zero. There might be a P wave propagating from $y$ into $\Omega_1$, not shown on the figure. We let them propagate and exit $\Omega_1$  through $\bo$. 

Since the tilded system has the same Cauchy data on $(0,T)\times\bo$ (as above, we shift the time, if needed so that the solution is smooth for $t<0$), and $c_p=\tilde c_p$, $c_s=\tilde c_2$ in $\Omega_1$ by Lemma~\ref{lemma_G1}, we get that the principal part of $u$ and $\tilde u$ coincide in the domain of influence. We recall that the principal parts do not depend on $\rho$. Then $u$ and $\tilde u$ have the same Cauchy data on $\Gamma_1^+$ as well. Then we can identify $y$ by the point where the first (in time) singularity hits $\Gamma_1$ again. The rest is as in the proof of the previous results. 


%

\end{document}